\documentclass[12pt]{article}

\usepackage{subfigure}

\usepackage{latexsym}
\usepackage{graphics}
\usepackage{amsmath}
\usepackage{hyperref}
\usepackage{xspace}
\usepackage{amssymb}
\usepackage{psfrag}
\usepackage{epsfig}
\usepackage{amsthm}
\usepackage{pst-all}

\usepackage{color}

\definecolor{Red}{rgb}{1,0,0}
\definecolor{Blue}{rgb}{0,0,1}
\definecolor{Olive}{rgb}{0.41,0.55,0.13}
\definecolor{Yarok}{rgb}{0,0.5,0}
\definecolor{Green}{rgb}{0,1,0}
\definecolor{MGreen}{rgb}{0,0.8,0}
\definecolor{DGreen}{rgb}{0,0.55,0}
\definecolor{Yellow}{rgb}{1,1,0}
\definecolor{Cyan}{rgb}{0,1,1}
\definecolor{Magenta}{rgb}{1,0,1}
\definecolor{Orange}{rgb}{1,.5,0}
\definecolor{Violet}{rgb}{.5,0,.5}
\definecolor{Purple}{rgb}{.75,0,.25}
\definecolor{Brown}{rgb}{.75,.5,.25}
\definecolor{Grey}{rgb}{.5,.5,.5}

\newcommand{\pr}{\mathbb{P}}

\newcommand{\R}{\mathbb{R}}
\newcommand{\Z}{\mathbb{Z}}

\newcommand{\distr}{\stackrel{d}{=}}

\newcommand{\N}{{\bf \mathcal{N}}}

\newcommand{\ignore}[1]{\relax}

\newtheorem{theorem}{Theorem}[section]
\newtheorem{remark}[theorem]{Remark}
\newtheorem{lemma}[theorem]{Lemma}
\newtheorem{conjecture}[theorem]{Conjecture}

\newtheorem{proposition}[theorem]{Proposition}

\newcommand{\ER}{Erd{\"o}s-R\'{e}nyi }

\definecolor{Red}{rgb}{1,0,0}
\definecolor{Blue}{rgb}{0,0,1}
\definecolor{Olive}{rgb}{0.41,0.55,0.13}
\definecolor{Green}{rgb}{0,1,0}
\definecolor{MGreen}{rgb}{0,0.8,0}
\definecolor{DGreen}{rgb}{0,0.55,0}
\definecolor{Yellow}{rgb}{1,1,0}
\definecolor{Cyan}{rgb}{0,1,1}
\definecolor{Magenta}{rgb}{1,0,1}
\definecolor{Orange}{rgb}{1,.5,0}
\definecolor{Violet}{rgb}{.5,0,.5}
\definecolor{Purple}{rgb}{.75,0,.25}
\definecolor{Brown}{rgb}{.75,.5,.25}
\definecolor{Grey}{rgb}{.5,.5,.5}
\definecolor{Pink}{rgb}{1,0,1}
\definecolor{DBrown}{rgb}{.5,.34,.16}
\definecolor{Black}{rgb}{0,0,0}

\usepackage{float}

\author{{\sf David Gamarnik}\thanks{MIT; e-mail: {\tt gamarnik@mit.edu}. Support from ONR Grant N00014-17-1-2790 is gratefully acknowledged.} \and
{\sf Eren C. K{\i}z{\i}lda\u{g}}\thanks{MIT; e-mail: {\tt kizildag@mit.edu}.}
}

\begin{document}

\title{Computing the partition function of the Sherrington-Kirkpatrick model is hard on average\footnote{This paper is a strengthened version of an earlier unpublished manuscript \cite{gamarnik2018computing}.}}
\date{\today}

\maketitle

\begin{abstract}
We establish the average-case hardness of the algorithmic problem of exact computation of the partition function associated with the Sherrington-Kirkpatrick model of spin glasses with Gaussian couplings and random external field. In particular, we establish that unless $P= \#P$, there does not exist a polynomial-time algorithm to exactly compute the partition function on average. This is done by showing that if there exists a polynomial time algorithm, which exactly computes the partition function for inverse polynomial fraction ($1/n^{O(1)}$) of all inputs, then there is a polynomial time algorithm, which exactly computes the partition function for all inputs, with high probability, yielding $P=\#P$. The computational model that we adopt is  {\em finite-precision arithmetic}, where the algorithmic inputs are truncated first to a certain level $N$ of digital precision. The ingredients of our proof include the random and downward self-reducibility of the partition function with random external field; an argument of Cai et al. \cite{cai1999hardness} for establishing the average-case hardness of computing the permanent of a matrix; a list-decoding algorithm of Sudan \cite{sudan1996maximum}, for reconstructing polynomials intersecting a given list of numbers at sufficiently many points; and near-uniformity of the log-normal distribution, modulo a large prime $p$. To the best of our knowledge, our result is the first one establishing a provable hardness of a model arising in the field of spin glasses.

Furthermore, we extend our result to the same problem under a different {\em real-valued} computational model, e.g. using a Blum-Shub-Smale machine \cite{blum1988theory} operating over real-valued inputs. We establish that, if there exists a polynomial time algorithm which exactly computes the partition function for $\frac34+\frac{1}{n^{O(1)}}$ fraction of all inputs, then there exists a polynomial time algorithm, which exactly computes the partition function for all inputs, with high probability, yielding $P=\#P$. Our proof uses the random self-reducibility of the partition function, together with a control over the total variation distance for log-normal random variables in presence of a convex perturbation, and the Berlekamp-Welch algorithm.
\end{abstract}
\pagebreak
\tableofcontents
\section{Introduction}
The subject of this paper is the algorithmic hardness of the problem of exactly computing the partition function associated with the Sherrington-Kirkpatrick (SK) model of spin glasses, a mean field model that was first introduced by Sherrington and Kirkpatrick in 1975 \cite{sherrington1975solvable}, to propose a solvable model for the 'spin-glass' phase, an unusual magnetic behaviour predicted to occur in spatially random physical systems. The model is as follows. Fix a positive integer $n$, and consider $n$ sites $i\in\{1,2,\dots,n\}$, a naming motivated from a site of a magnet. To each site $i$, assign a spin, $\sigma_i\in\{-1,1\}$, and define the energy Hamiltonian $H(\boldsymbol{\sigma})$ for this spin configuration $\boldsymbol{\sigma}=(\sigma_i:1\leq i\leq n)\in\{-1,1\}^n$ via $H(\boldsymbol{\sigma})=\frac{\beta}{\sqrt{n}}\sum_{1\leq i<j\leq n}J_{ij}\sigma_i\sigma_j$,
where the parameters ${\bf J}=(J_{ij}:1\leq i<j\leq n)\in \mathbb{R}^{n(n-1)/2}$ are called spin-spin interactions (or shortly, couplings), and the parameter $\beta$ is called the inverse temperature. The associated partition function is given by, $Z({\bf J,\beta}) = \sum_{\boldsymbol{\sigma}\in\{-1,1\}^n}\exp\left(-\frac{\beta}{\sqrt{n}}\sum_{1\leq i<j\leq n}J_{ij}\sigma_i\sigma_j\right)$.
The SK model corresponds to the case, where the couplings $J_{ij}$ are iid standard normal; and the partition function, $Z({\bf J,\beta})$ carries useful information about the underlying physical system. The SK model is a \emph{mean-field} model of spin glasses, namely the interaction between any two distinct sites, $1\leq i<j\leq n$, is modeled  with random coupling parameters $J_{ij}$, which do not depend on the spatial location of $i$ and $j$. The rationale behind the scaling $\sqrt{n}$ is to ensure that the average energy per spin is roughly independent of $n$, and consequently, the free energy limit, $\lim_n n^{-1}\log Z({\bf J,\beta})$ is non-trivial. 
Despite the simplicity of its formulation, it turns out that the SK model is highly non-trivial to study, and that, analyzing the behaviour of a more elaborate model (such as, a model where the spatial positions of the sites are incorporated, by modelling them as the vertices of $\mathbb{Z}^2$ and the couplings are modified to be position-dependent) is really difficult. For a more detailed discussion on these and related issues, see the monographs by Panchenko \cite{panchenko2013sherrington} and Talagrand \cite{talagrand2010mean}.

In the first part of this paper, we focus on the SK model with the (random) external field, which was studied by Talagrand \cite{talagrand2010mean} (equation 1.61 therein), namely, the model, where the energy Hamiltonian is given by,
\begin{equation}\label{eqn:external-model}
H(\boldsymbol{\sigma})=\frac{\beta}{\sqrt{n}}\sum_{1\leq i<j\leq n}J_{ij}\sigma_i\sigma_j + \sum_{i=1}^n A_i\sigma_i.
\end{equation}
Here, the iid standard normal random variables ${\bf J}=(J_{ij}:1\leq i<j\leq n)\in\R^{n(n-1)/2}$ are the couplings, and the independent zero-mean normal random variables, ${\bf A}=(A_i:i\in[n])\in \R^n$ incorporate the external field contribution. To address this model we study the following equivalent model with the energy Hamiltonian:
\begin{equation}\label{eqn:model}
H(\boldsymbol{\sigma})=\frac{\beta}{\sqrt{n}}\sum_{1\leq i<j\leq n}J_{ij}\sigma_i\sigma_j + \sum_{i=1}^n B_i\sigma_i -\sum_{i=1}^n C_i \sigma_i,
\end{equation}
where ${\bf J}=(J_{ij}:1\leq i<j\leq n)\in \R^{n(n-1)/2}$ are the couplings as above, ${\bf B}=(B_i:i\in[n])\in\R^n$ and ${\bf C}=(C_i:i\in[n])\in\R^n$ are independent zero-mean normal random variables, which we still refer to as external field components.
Observe that, if $\mathcal{A}_1$ is an oracle, which, for input $({\bf J,A})$, computes the partition function for the model whose Hamiltonian is given by (\ref{eqn:external-model}), then $\mathcal{A}_1$ with input $({\bf J,B-C})$ computes the partition function of the model whose Hamiltonian is given by (\ref{eqn:model}). Similarly, if $\mathcal{A}_2$ is an oracle, which, for input $({\bf J,B,C})$ computes the partition function of the model in (\ref{eqn:model}), then $\mathcal{A}_2$ with input, $({\bf J,\frac{A+G}{2},\frac{G-A}{2}})$, where ${\bf G}=(G_i:i\in [n])$ is an iid copy of the random vector ${\bf A}=(A_i:i\in[n])$, computes the partition function of the model in (\ref{eqn:external-model}), recalling that if $A_i$ and $G_i$ are iid Gaussian random variables, then $A_i+G_i$ and $A_i-G_i$ are independent. 
In spite of being equivalent, the model in (\ref{eqn:model}), however, is more convenient to work with, in particular, for establishing a a certain downward self-recursive formula which expresses the partition function of an $n$-spin SK model as a weighted sum of the partition functions of two $(n-1)$-spin SK models, with properly adjusted external field components. 

The algorithmic problem is the problem of computing the partition function $Z({\bf J,B,C})$ associated to the modified model in (\ref{eqn:model}), when $({\bf J,B,C})\in\mathbb{R}^{n(n-1)/2+2n}$ is given as a (random) input. The (worst-case) algorithmic problem of computing $Z({\bf J,B,C})$ for an arbitrary input $({\bf J,B,C})$ is known to be \#P-hard for a much broader class of statistical physics models and associated partition functions, see e.g. \cite{barahona1982computational} and \cite{istrail2000statistical}. On the other hand, the classical reduction techniques that are used for establishing worst-case hardness do not seem to transfer to the problems with random inputs. The subject of this paper is the case of Gaussian random inputs, $({\bf J,B,C})$. The computational model that we adopt in the first part of the paper is the finite-precision arithmetic, and therefore the real-valued vector $({\bf J,B,C})$ cannot be used as a formal algorithmic input. In order to handle this issue, we consider a model, where the algorithm designer first selects a level $N$ of digital precision, and the values of $J_{ij},B_i,C_i$, or more concretely, $\widehat{J}_{ij}=\exp(\frac{\beta J_{ij}}{\sqrt{n}}),\widehat{B}_{i}=\exp(B_i)$, and $\widehat{C}_{i}=\exp(C_i)$ are computed, up to this selected level $N$ of digital precision: $\widehat{J}_{ij}^{[N]}, \widehat{B}_i^{[N]}$, and $\widehat{C}_i^{[N]}$, where $x^{[N]}=2^{-N}\lfloor 2^N x\rfloor$. The task of the algorithm designer is to exactly compute the partition function, associated with the input $(\widehat{J}_{ij}^{[N]}:1\leq i<j\leq n)$, $(\widehat{B}_i^{[N]}:1\leq i\leq n)$, and $(\widehat{C}_i^{[N]}:1\leq i\leq n)$ in polynomial (in $n$) time.

The main result of under the aforementioned assumptions is as follows. Let $k>0$ be any arbitrary constant. If there exists a polynomial time algorithm, which computes the partition function exactly with probability at least $1/n^k$, then $P=\#P$. Here, the  probability is taken with respect to the randomness of $({\bf J,B,C})$. To the best of our knowledge, this is the first result establishing formal algorithmic hardness of a computational problem arising in the field of spin glasses.

The approach we pursue here aims at capturing a \emph{worst-case} to \emph{average-case} reduction, and is similar to  establishing the average-case hardness of other problems involving counting, such as the problem of counting cliques in \ER hypergraphs \cite{adsera2019average}, or the problem of computing the permanent of a matrix modulo $p$, with entries chosen independently and uniformly over finite field $\mathbb{Z}_p$. Lipton observed in \cite{lipton1989new} that, for  a suitably chosen prime $p$, the permanent of a matrix can be expressed as a univariate polynomial, generated using integer multiples of a random uniform input. Hence, provided this polynomial can be recovered, the permanent of any \emph{arbitrary} matrix can be computed. Therefore, the average-case hardness of computing the permanent of a matrix modulo $p$ equals the worst-case hardness of the same problem, which is known to be \#P-hard. Lipton proves his result, by assuming there exists an algorithm, which correctly computes the permanent for at least $1-O(1/n)$ fraction of matrices over $\mathbb{Z}_p^{n\times n}$. Subsequent research weakened this assumption to the existence of an algorithm with constant probability of success \cite{feige1992hardness}, and finally, to the existence of an algorithm with inverse polynomial probability ($1/n^{O(1)}$) of success, Cai et al. \cite{cai1999hardness}, a regime, which is also our focus.
The proof technique that we follow is similar to that of Cai et al. \cite{cai1999hardness}
, and is built upon earlier ideas from Gemmell and Sudan \cite{gemmell1992highly}, Feige and Lund \cite{feige1992hardness}, and Sudan \cite{sudan1996maximum}. 

More specifically, the argument of Cai et al. \cite{cai1999hardness} is as follows. The permanent of a given matrix $M\in\mathbb{Z}_p^{n\times n}$ equals, via Laplace expansion, a weighted sum of the permanents of $n$ minors $M_{11},M_{21},\dots,M_{n1}$ of $M$, each of dimension $n-1$. Then a certain matrix polynomial is constructed, whose value at $k$ is equal to $M_{k1}$, by incorporating two random matrices, independently generated from the uniform distribution on $\mathbb{Z}_p^{(n-1)\times (n-1)}$. The permanent of this matrix polynomial is a univariate polynomial over a finite field with a known upper bound on its degree, and the problem boils down to recovering this polynomial from a list of pairs of numbers intersecting the graph of the polynomial at sufficiently many points. This, in fact, is a standard problem in coding theory, and the recovery of this polynomial is achieved by a list-decoding algorithm by Sudan \cite{sudan1996maximum}, which is an improved version of Berlekamp-Welch decoder. 

The method that we use follows the proof technique of Cai et al. \cite{cai1999hardness}, with several additional modifications. First, to avoid dealing with correlated random inputs, we reduce the problem of computing the partition function of the model in (\ref{eqn:model}) to computing the partition function of a different object, where the underlying cuts and polarities induced by the spin assignment $\boldsymbol{\sigma}\in\{-1,1\}^n$ are  incorporated. 
Second, a downward self-recursion formula for computing the partition function, analogous to Laplace expansion for permanent, is established; and this is the rationale for using the aforementioned equivalent model whose Hamiltonian is given by (\ref{eqn:model}). This is achieved by recursing downward with respect to the sign of $\sigma_n$, and expressing the partition function of an $n$-spin system, with a weighted sum of the partition functions of two $(n-1)$-spin systems, with appropriately adjusted external field components. Third, recalling that, we are interested in the case of random Gaussian inputs, we establish a probabilistic coupling between truncated version of log-normal distribution, and uniform distribution modulo a large prime $p$. Towards this goal, we establish that the log-normal distribution is "sufficiently" Lipschitz in a small interval and near-uniform modulo $p$. Finally, we also need to connect modulo $p$ computation to the exact computation of the partition function, in the sense defined above, i.e., truncating the inputs up to a certain level $N$  of digital precision, and computing the associated partition function. This is  achieved by using a  standard Chinese remaindering argument: Take  prime numbers $p_1,\dots,p_K$, compute $Z\pmod{p_i}$, for every $i$, and use this information to compute  compute $Z\pmod{P}$ where $P=\prod_{k=1}^K p_k$, via Chinese  remaindering. Provided $P>Z$, $Z\pmod{p}$ is precisely $Z$. The existence of sufficiently many such primes of appropriate size that we can work with is justified through the prime number theorem.

In the second part of the paper, we focus on the same problem without the external field component, but this time under the {\em real-valued} computational model. We recall the model for convenience. First, generate iid standard normal random variables, $J_{ij}$; and let the elements of the sequence ${\bf J}=(J_{ij}:1\leqslant i<j\leqslant n)\in \R^{\frac{n(n-1)}{2}}$ be the couplings. For each spin configuration $\boldsymbol{\sigma}=(\sigma_i:1\leqslant i\leqslant n)\in \{-1,1\}^n$, define the associated energy Hamiltonian  $H(\boldsymbol{\sigma})=\sum_{i<j}J_{ij}\sigma_i\sigma_j$. The algorithmic question of interest is the exact computation of the associated partition function, namely the object,
$$
Z({\bf J}) = \sum_{\boldsymbol{\sigma}\in\{-1,1\}^n} \exp\left(-\sum_{1\leqslant i<j\leqslant n}J_{ij}\sigma_i\sigma_j\right) = \sum_{\boldsymbol{\sigma}\in\{-1,1\}^n}\exp(-H(\boldsymbol{\sigma})),
$$
using the real-valued computational model, {\em e.g.} a Blum-Shub-Smale machine \cite{blum1988theory} operating over real-valued inputs, as opposed to the previous setting, where the computational engine performs floating point operations. 
The input vector, namely the vector of real-valued couplings ${\bf J}\in\R^{n(n-1)/2}$, is given as a random input. Albeit the usual definition of the partition function involves also the inverse temperature parameter $\beta$, and a normalization factor by $\sqrt{n}$; we suppress these in order to keep the discussion simple. 

The main result towards this direction is as follows. If there exists a polynomial time algorithm, which computes the partition function exactly with probability at least $3/4+1/{\rm poly}(n)$ under real-valued computational model, then $P=\#P$. Similar to the previous setting, the probability here is taken with respect to the randomness in the input of the algorithm, namely, with respect to the distribution of ${\bf J}$. 

The techniques of the previous setting (finite-precision arithmetic) do not, however, transform to real-valued computational model, since the finite field structure $\mathbb{Z}_p$ utilized for the proof is lost, upon passing to real-valued computation model. We bypass this obstacle by building on the argument of Aaronson and Arkhipov \cite{aaronson2011computational}, where they established the average-case hardness of the exact computation of the permanent of a random matrix with iid Gaussian entries.

We close this section with the set of notational convention. The set of integers and positive integers are respectively denoted by $\mathbb{Z}$ and $\mathbb{Z}^+$. The set, $\{1,2,\dots,n\}$ is denoted by $[n]$, and the set $\{0,1,\dots,p-1\}$, namely the set of all residues modulo $p$, is denoted by $\mathbb{Z}_p$. Given a real number $x$, the largest integer not exceeding $x$ is denoted by $\lfloor x\rfloor$. We say $a\equiv b\pmod{p}$, if $p$ divides $a-b$, abbreviated as $p\mid a-b$. Given an $x>0$, $\log x$ denotes logarithm of $x$, base $2$. Given a (finite) set $S$, denote the number of elements (i.e., the cardinality) of $S$  by $|S|$. Given a finite field $\mathbb{F}$, denote by $\mathbb{F}[x]$ the set of all (finite-degree) polynomials, whose coefficients are from $\mathbb{F}$. Namely, $f\in\mathbb{F}[x]$ if there is a positive integer $n$, and $a_0,\dots,a_n\in\mathbb{F}$, such that for every $x\in\mathbb{F}$, $f(x)=\sum_{k=0}^n a_k x^k$. The degree of $f\in\mathbb{F}[x]$ is, ${\rm deg}(f)=\max\{0\leq k\leq n:a_k\neq 0\}$. For two random variables $X$ and $Y$, the total variation distance between (the distribution functions of) $X$ and $Y$ is denoted by $d_{TV}(X,Y)$. For any given vector $v\in\R^d$, we denote by $\|v\|$ the Euclidean norm of $v$, that is, $\sqrt{\sum_{i=1}^d v_i^2}$. $\Theta(\cdot),O(\cdot)$, $o(\cdot)$, and $\Omega(\cdot)$ are standard (asymptotic) order notations for comparing the growth of two sequences. Finally, we use the words oracle and algorithm interchangeably in the sequel, and denote them by $\mathcal{O}$ and $\mathcal{A}$. These objects will be assumed to exist for the sake of proof purposes.

\section{Average-Case Hardness under Finite-Precision Arithmetic}
\subsection{Model and the Main Result}
Our focus is on computing the partition function of the model, whose Hamiltonian for a given spin configuration $\boldsymbol{\sigma}\in\{-1,1\}^n$ at inverse temperature $\beta$ is given by:
\begin{align*}
    H(\boldsymbol{\sigma})= \frac{\beta}{\sqrt{n}}\sum_{1\leq i<j\leq n}J_{ij}\sigma_i \sigma_j+\sum_{i=1}^n B_i\sigma_i-\sum_{i=1}^n C_i \sigma_i,
\end{align*}
where the random variables ${\bf J}=(J_{ij}:1\leq i<j\leq n)\in \R^{n(n-1)/2}$ are the couplings; and the random variables ${\bf B}=(B_i:i\in[n])\in \R^n$, and ${\bf C}=(C_i:i\in[n])\in \R^n$ are the external field components. For simplicity, we study the case, where $({\bf J,B,C})$ consists of iid standard normal entries. However, our analysis also applies to the case, where $({\bf J,B,C})$ consists of independent normal entries with zero-mean and possibly different variances. Associated partition function (at the temperature $1/\beta$) reads as:
$$
Z({\bf J,B,C})=\sum_{\boldsymbol{\sigma}\in\{-1,1\}^n}\exp(-H(\boldsymbol{\sigma})).
$$
We now incorporate the cuts and polarities induced by $\boldsymbol{\sigma}\in\{-1,1\}^n$. Observe that,
\begin{align*}
    H(\boldsymbol{\sigma})
    &= \frac{\beta}{\sqrt{n}}\sum_{i<j,\sigma_i=\sigma_j}J_{ij}+\sum_{i=1,\sigma_i=+1}^n B_i+ \sum_{i=1,\sigma_i=-1}^n C_i 
    &-\left(\frac{\beta}{\sqrt{n}} \sum_{i<j,\sigma_i\neq \sigma_j}J_{ij}+\sum_{i=1,\sigma_i=-1}^n B_i+ \sum_{i=1,\sigma_i=+1}^n C_i\right).
\end{align*}
For convenience, we will denote the first part above by $\Sigma_{\boldsymbol{\sigma}}^+$, and the second part inside the brackets by $\Sigma_{\boldsymbol{\sigma}}^-$. Observe that, the object, $\Sigma\triangleq \sum_i B_i+\sum_i C_i+\frac{\beta}{\sqrt{n}}\sum_{i<j}J_{ij}=\Sigma_{\boldsymbol{\sigma}}^+ + \Sigma_{\boldsymbol{\sigma}}^-$, is independent of $\boldsymbol{\sigma}$, and trivially computable. Now, note that, $\Sigma-H(\boldsymbol{\sigma})=2\Sigma_{\boldsymbol{\sigma}}^-$, and therefore,
$$
Z({\bf J,B,C})=\sum_{\boldsymbol{\sigma}\in\{-1,1\}^n} \exp(-\Sigma)\exp(2\Sigma_{\boldsymbol{\sigma}}^-).
$$
Namely, $Z({\bf J,B,C})$ is computable, if and only if, $\sum_{\boldsymbol{\sigma}\in\{-1,1\}^n} \exp(2\Sigma_{\boldsymbol{\sigma}}^-)$ is computable. The presence of the factor $2$ is, again, a minor detail that we omit in the sequel, since our techniques transfer without any modification. Thus, our focus is on computing $\sum_{\boldsymbol{\sigma}\in\{-1,1\}^n} \exp(\Sigma_{\boldsymbol{\sigma}}^{-})$; and denoting $\exp(\beta J_{ij}/\sqrt{n})$ by $\widehat{J_{ij}}$, $\exp(B_i)$ by $\widehat{B_i}$, and $\exp(C_i)$ by $\widehat{C_i}$, the object we are interested in computing is given by,
$$
Z({\bf \widehat{J},\widehat{B},\widehat{C}})=\sum_{\boldsymbol{\sigma}\in \{-1,1\}^n}\left(\prod_{i:\sigma_i=-}\widehat{B_i} \right)\left(\prod_{i:\sigma_i=+}\widehat{C_i}\right)\left(\prod_{i<j:\sigma_i\neq \sigma_j}\widehat{J_{ij}}\right).
$$
Our focus is on algorithms, that can compute $Z({\bf \widehat{J},\widehat{B},\widehat{C}})$ exactly, in the following sense. The algorithm designer first selects a certain level $N$ of digital precision, and computes these numbers, up to the selected precision level. Given a real number $x\in\mathbb{R}$, let $x^{[N]} = 2^{-N}\lfloor 2^N x\rfloor$ be the number obtained by keeping only first $N$ binary bits of $x$ after the binary point. The computational goal of the algorithm designer is to compute $Z({\bf \widehat{J}^{[N]},\widehat{B}^{[N]},\widehat{C}^{[N]}})$ exactly, where ${\bf \widehat{J}^{[N]}}=(\widehat{J_{ij}}^{[N]}:1\leq i<j\leq n)$, ${\bf \widehat{B}^{[N]}} = (\widehat{B_i}^{[N]}:i\in[n])$, and ${\bf \widehat{C}^{[N]}}= (\widehat{C_i}^{[N]}:i\in[n])$.

We now switch to a model with integer inputs. For convenience, let $\widetilde{J_{ij}}=\lfloor 2^N \widehat{J_{ij}} \rfloor= 2^N\widehat{J_{ij}}^{[N]}$, $\widetilde{B_i}=\lfloor 2^N \widehat{B_i} \rfloor= 2^N\widehat{B_i}^{[N]}$,  $\widetilde{C_i}=\lfloor 2^N \widehat{C_i} \rfloor= 2^N\widehat{C_i}^{[N]}$; and define $f(n,\boldsymbol{\sigma})$ to be
\begin{equation}\label{eq:f-of-n-sigma}
f(n,\boldsymbol{\sigma})=\frac{n(n-1)}{2}-n-I_n(\boldsymbol{\sigma}),
\end{equation}
where $I_n(\boldsymbol{\sigma})=|\{(i,j):\sigma_i\neq \sigma_j,1\leq i<j\leq n\}|$. Equipped with this, we will focus on computing the following object with integer-valued inputs,
\begin{equation}\label{eqn:main-pf}
Z_n({\bf \widetilde{J},\widetilde{B},\widetilde{C}})=\sum_{\boldsymbol{\sigma}\in \{-1,1\}^n}2^{Nf(n,\boldsymbol{\sigma})}\left(\prod_{i:\sigma_i=-}\widetilde{B_i} \right)\left(\prod_{i:\sigma_i=+}\widetilde{C_i}\right)\left(\prod_{i<j:\sigma_i\neq \sigma_j}\widetilde{J_{ij}}\right),
\end{equation}
where the subscript $n$ highlights the dependence on $n$, indicating that the system consists of $n$ spins. Observe that, $Z_n({\bf \widetilde{J},\widetilde{B},\widetilde{C}})=2^{Nn(n-1)/2} Z({\bf \widehat{J}^{[N]},\widehat{B}^{[N]},\widehat{C}^{[N]}})$. As a sanity check, note that $|I_n(\boldsymbol{\sigma})|+n\leq \max_{0<k<n} k(n-k) +n <n(n-1)/2$ for $n>6$, and every $\boldsymbol{\sigma}\in\{-1,1\}^n$. Thus the model is indeed integral-valued. 

We now state our main result, for the average-case hardness of computing $Z_n({\bf \widetilde{J},\widetilde{B},\widetilde{C}})$.

\begin{theorem}\label{thm:main}
Let $k,\alpha>0$ be arbitrary fixed constants. Suppose that, the precision value $N$ satisfies, $C(\alpha,k) \log n\leq N \leq n^\alpha$, where $C(\alpha,k)$ is a constant, depending only on $\alpha$ and $k$. Suppose that there exists a polynomial in $n$ time algorithm $\mathcal{A}$, which on input $({\bf \widetilde{J},\widetilde{B},\widetilde{C}})$ produces a value $Z_{\mathcal{A}}({\bf \widetilde{J},\widetilde{B},\widetilde{C}})$ satisfying
$$
\mathbb{P}(Z_{\mathcal{A}}({\bf \widetilde{J},\widetilde{B},\widetilde{C}})=Z_n({\bf \widetilde{J},\widetilde{B},\widetilde{C}}))\geq \frac{1}{n^k},
$$
for all sufficiently large $n$, where $Z_n({\bf \widetilde{J},\widetilde{B},\widetilde{C}})$ is defined in (\ref{eqn:main-pf}). Then, $P=\#P$.
\end{theorem}
Quantitatively, the constant $C(\alpha,k)$ can be taken as $3\alpha+21k/2+10+\epsilon$, where $\epsilon>0$ is arbitrary. The probability in Theorem \ref{thm:main}  is taken with respect to the randomness of $({\bf \widetilde{J},\widetilde{B},\widetilde{C}})$, which, in turn, is derived from the randomness of $({\bf J,B,C})$. The logarithmic lower bound on the number of bits is imposed to address the technical issues when establishing the near-uniformity of the random variables $({\bf \widetilde{J},\widetilde{B},\widetilde{C}})$ modulo an appropriately chosen prime. The upper bound on the number of bits that we retain is for ensuring that the input to the algorithm is of polynomial length in $n$.

\subsection{Proof of Theorem \ref{thm:main}}
For any given ${\bf \Xi}\in\mathbb{Z}^{n(n-1)/2+2n}$ (note that, any algorithm computing the partition function of an $n$-spin system with external field accepts an input of size $n(n-1)/2 + 2n$), let $Z_n({\bf \Xi},p_n)\in\mathbb{Z}_{p_n}$ denotes $Z_n({\bf \Xi})\pmod{p_n}$, and similarly let $Z_{\mathcal{A}}({\bf \Xi};p_n)$ denotes $Z_{\mathcal{A}}({\bf \Xi})\pmod{p_n}$. Let ${\bf U}\in\mathbb{Z}_{p_n}^{n(n-1)/2+2n}$ be a random vector, consisting of iid entries, drawn independently  from uniform distribution on $\mathbb{Z}_{p_n}$. 
The following result is  our main proposition, and establishes the average-case hardness of computing the partition function defined in (\ref{eqn:main-pf}) modulo $p_n$, when the entry to the algorithm is ${\bf U}$. 
This, together with a coupling argument will establish Theorem \ref{thm:main}.

\begin{proposition}\label{prop-main}
Let $k>0$ be an arbitrary constant. Suppose $\mathcal{A}$ is a polynomial in $n$ time algorithm, which for any positive integer $n$, any prime number $p_n\geq 9 n^{2k+2}$, and any input ${\bf a}=({\bf a^J,a^B,a^C})\in \mathbb{Z}_{p_n}^{n(n-1)/2 + 2n}$ produces some output $Z_{\mathcal{A}}({\bf a};p_n) \in \mathbb{Z}_{p_n}$; and satisfies
$$
\mathbb{P}(Z_{\mathcal{A}}({\bf U};p_n)=Z_n({\bf U};p_n))\geq \frac{1}{n^k},
$$
where ${\bf U}=({\bf U^J, U^B,U^C})\in\mathbb{Z}_{p_n}^{n(n-1)/2+2n}$ consists of iid entries chosen uniformly at random from $\mathbb{Z}_{p_n}$, and the probability is taken with respect to the randomness in ${\bf U}$. Then, $P=\#P$.
\end{proposition}
\begin{proof}{(of Proposition \ref{prop-main})}
We will use as basis the \#P-hardness of computing the partition function, for arbitrary inputs. Namely, if there exists a polynomial time algorithm computing $Z({\bf j,b,c})$ for any arbitrary input ${\bf j,b,c}$ with probability bounded away from zero as $n\to\infty$, then P=\#P. 

Let $q\geq 1/n^k$ be the success probability of $\mathcal{A}$, and ${\bf a}=({\bf  a^J,a^B,a^C})\in\mathbb{Z}_{p_n}^{n(n-1)/2+2n}$ be an \emph{arbitrary} input, whose partition function we want to compute. For convenience, we drop  ${\bf a}$,  denote $({\bf a^J,a^B,a^C})$ by $({\bf J,B,C})$. The following lemma establishes the downward self-recursive behaviour of the partition function (modulo $p_n$) by expressing the partition function of an $n$-spin system as a weighted sum of partition functions of two $(n-1)$-spin systems, with appropriately adjusted external field components.
\begin{lemma}\label{lemma:self-recursion}
The following identity holds:
$$
Z_n({\bf J,B,C};p_n)=C_n' Z_{n-1}({\bf J',B^+,C^+};p_n) + B_n' Z_{n-1}({\bf J',B^-,C^-};p_n),
$$
where, $Z_n({\bf J,B,C};p_n)=Z_n({\bf J,B,C})\pmod{p_n}$ with $Z_n$ defined in (\ref{eqn:main-pf}); ${\bf J'}\in\mathbb{Z}_{p_n}^{(n-1)(n-2)/2}$ is such that $J'_{ij}=J_{ij}$ for every $1\leq i<j\leq n-1$; ${\bf B^+,B^-,C^+,C^-}\in \mathbb{Z}_{p_n}^{n-1}$ are such that $B^+_i = 2^{-N}B_i J_{in}$, $B^-_i = B_i$, $C^+_i=C_i$, and $C^-_i = 2^{-N}C_iJ_{in}$, for every $1\leq i\leq n-1$; and $C_n'=C_n2^{(n-2)N}$, $B_n'=B_n 2^{(n-2)N}$.
\end{lemma}
The proof of this lemma is provided in Section \ref{subsec:self-recursion}. Namely, provided we can compute $Z_{n-1}({\bf J',B^+,C^+};p_n)$ and $Z_{n-1}({\bf J',B^-,C^-};p_n)$, we can compute $Z_n({\bf J,B,C};p_n)$. Note that, since we are interested in modulo $p_n$ computation, the number $2^{-N}$ is nothing but $g^N$, where $g\in \mathbb{Z}_{p_n}$ satisfies $2g\equiv 1\pmod{p_n}$, that is, $g$ is the multiplicative inverse of $2$ modulo $p_n$. 

Next, let $v_1=({\bf J',B^+,C^+})\in\mathbb{Z}_{p_n}^{T}$, and $v_2=({\bf J',B^-,C^-})\in\mathbb{Z}_{p_n}^{T}$; where the input dimension $(n-1)(n-2)/2+2(n-1)$ of the algorithm computing partition function for a model with $(n-1)$-spins  is denoted by $T$ for convenience. Now, we construct the vector polynomial 
\begin{equation}\label{eqn:d-of-x}
D(x)=(2-x)v_1+(x-1)v_2+(x-1)(x-2)(K+xM),
\end{equation}
of dimension $T$, where $K,M$ are iid random vectors, drawn from uniform distribution on $\mathbb{Z}_{p_n}^{T}$. The incorporation of this extra randomness is due to an earlier idea by Gemmell and Sudan \cite{gemmell1992highly}. 

Next, consider 
$\phi(x)=Z_{n-1}(D(x);p_n)
$, namely, the partition function of an $(n-1)$-spin system, associated with the vector $D(x)$ (where the first $(n-1)(n-2)/2$ components correspond  to couplings, the following $(n-1)$ components correspond  to $B_i$'s, and the last $(n-1)$ components correspond to $C_i$'s), which is a univariate polynomial in $x$. We now upper bound the degree of $\phi(x)$. Note that,
$$
d={\rm deg}(\phi)\leq 3\left(\max_{\boldsymbol{\sigma}\in\{-1,1\}^{n-1}}|I(\boldsymbol{\sigma})| +n-1 \right)=3\left(\max_{1\leq k\leq n-1}k(n-1-k)+n-1\right)<n^2,
$$
for $n$ large. Observe also that, $\phi(1)=Z_{n-1}(D(1);p_n)=Z_{n-1}({\bf J',B^+,C^+};p_n)$,  $\phi(2)=Z_n(D(2);p_n)=Z_{n-1}({\bf J',B^-,C^-};p_n)$, hence, $Z_n({\bf J,B,C};p_n)=C_n'\phi(1)+B_n' \phi(2)$. Therefore, provided that we can recover $\phi(\cdot)$, $Z_n({\bf J,B,C};p_n)$ can be computed. With this, we now turn our attention to recovering the polynomial $\phi(\cdot)$. Let $\mathcal{D}$ be a set of cardinality $p_n-2$, defined as $\mathcal{D}=\{D(x):x=3,4,\dots,p_n\}$. We claim that, $\mathcal{D}$ consists of pairwise independent samples.
\begin{lemma}\label{pairwise-Indep}
For every distinct $x_1, x_2\in\{3,4,\dots,p_n\}$, the random vectors $D(x_1)$ and $D(x_2)$ are independent and uniformly distributed over $\mathbb{Z}_{p_n}^T$. That is, for every such $x_1,x_2$ and every $y_1,y_2\in\mathbb{Z}_{p_n}^{T}$; it holds that $\mathbb{P}(D(x_1)=y_1)=1/p_n^T=\mathbb{P}(D(x_2)=y_2)$, and
$$
\mathbb{P}(D(x_1)=y_1,D(x_2)=y_2)=1/p_n^{2T}=\mathbb{P}(D(x_1)=y_1)\mathbb{P}(D(x_2)=y_2),
$$
where the probability is taken with respect to the randomness in $K$ and $M$.
\end{lemma}
The proof of this lemma is provided in Section \ref{subsec:lemma-pairwise-indep}. Now, we run $\mathcal{A}$ on $\mathcal{D}$, and will use the independence to deduce via Chebyshev's inequality that, with high probability, $\mathcal{A}$ runs correctly, on at least $q/2$ fraction of inputs in $\mathcal{D}$, where $q\geq 1/n^k$ is the success probability of our algorithm. This is encapsulated by the following lemma.
\begin{lemma}\label{whp}
Let the random variable $\mathcal{N}$ be the number of points $D(x)\in\mathcal{D}$, such that $\mathcal{A}(D(x))=\phi(x)=Z_{n-1}(D(x);p_n)$, namely $\mathcal{A}$ correctly computes the partition function at $D(x)$. Then,
$$
\mathbb{P}(\N\geq (p_n-2)q/2)\geq 1-\frac{1}{(p_n-2)q^2},
$$
where $q$ is the success probability of $\mathcal{A}$, and the probability is taken with respect to the randomness in $\mathcal{D}$, which, in turn, is due to the randomness in $K$ and $M$.
\end{lemma}
The proof of this lemma can be found in Section \ref{subsec:lemma-whp}. Now, let $G(f)=\{(x,f(x)):x=1,2,\dots,p_n\}$ be the graph of a function $f\in\mathbb{Z}_{p_n}[x]$. Define the set $\mathcal{S}=\{(x,\mathcal{A}(D(x))):x=3,4,\dots,p_n\}$, and let $\mathcal{F}$ be a set of polynomials, defined as,
$$
\mathcal{F}=\{f\in\mathbb{Z}_{p_n}[x]:{\rm deg}(f)<n^2,|G(f)\cap \mathcal{S}|\geq (p_n-2)q/2\}.
$$
Namely, $f\in\mathcal{F}$ if and only if, its coefficients are from $\mathbb{Z}_{p_n}$, it is of degree at most $n^2-1$; and its graph intersects the set $\mathcal{S}$ on at least $(p_n-2)q/2$ points. Due to Lemma \ref{whp}, we know that $\phi(x)\in \mathcal{F}$, with probability at least $1-\frac{1}{(p_n-2)q^2}$. 
We now show that this set $\mathcal{F}$ of candidate polynomials contains at most polynomial in $n$ many polynomials. 
\begin{lemma}\label{poly-in-n-F}
If $p_n\geq 9n^{2k+2}$, then $|\mathcal{F}|\leq  3/q$, where $q$ is the success probability of $\mathcal{A}$. In particular, $\mathcal{F}$ contains at most polynomial in $n$ many polynomials, since $q\geq 1/n^k$.
\end{lemma}
The proof of this lemma is provided in Section \ref{subsec-poly-in-n-F}.

In what remains, we will show how to explicitly construct all such polynomials, through a randomized algorithm, which succeeds with high probability. To that end, we use the following elegant result, due to Cai et al. \cite{cai1999hardness}. 
\begin{lemma}\label{lemma:list-decod}
There exists a randomized procedure running in polynomial time, through which, with high probability, one can generate a list $\mathcal{L}=(x_i,y_i)_{i=1}^L$ of $L$ pairs, such that, $y_i=\phi(x_i)$, for at least $t$ pairs from the list with distinct first coordinates, where $t>\sqrt{2Ld}$, with $d={\rm deg}(\phi)$, and $\phi(x)=Z_{n-1}(D(x);p_n)$.
\end{lemma}
The proof of this lemma is isolated from the argument of \cite{cai1999hardness}, and provided in Section \ref{subsec:list-decod} for completeness. 
Of course, these discussions are all based on the assumption that we condition on the high probability event that $\{\N\geq (p_n-2)q/2\}$, where $\N$ is the random variable defined in Lemma \ref{whp}. 

Having obtained this list, we now turn our attention to finding all polynomials (where, by Lemma \ref{poly-in-n-F}, there is at most polynomial in $n$ many of those), whose graph intersects the list at at least $t$ points with distinct first coordinates (for the specific values of $t$ depending on the magnitude of $p_n$, see the proof of Lemma \ref{lemma:list-decod} in Section \ref{subsec:list-decod}). For this, we use the following list-decoding algorithm of \cite{sudan1996maximum}, introduced originally in the context of coding theory, which is an improved version of Berlekamp-Welch decoder.
\begin{lemma}{(Theorem 5 in \cite{sudan1996maximum})}\label{madhu}
Given a sequence $\{(x_i,y_i)\}_{i=1}^L$ of $L$ distinct pairs, where $x_i$s and $y_i$s are an element of a field $\mathbb{F}$, and integer parameters $t$ and $d$, such that $t\geq d\lceil \sqrt{2(L+1)/d}\rceil-\lfloor d/2\rfloor$, there exists an algorithm which can find all polynomials $f:\mathbb{F}\to \mathbb{F}$ of degree at most $d$, such that the number of points $(x_i,y_i)$ satisfying $y_i=f(x_i)$ is at least $t$.
\end{lemma}
The algorithm is a probabilistic polynomial time algorithm. For the sake of completeness, we briefly sketch his algorithm here. For weights $w_x,w_y\in\mathbb{Z}^+$, define $(w_x,w_y)$-weighted degree of a monomial $q_{ij}x^i y^j$ to be $iw_x+jw_y$. The $(w_x,w_y)$-weighted degree of a polynomial, $Q(x,y)=\sum_{(i,j)\in I}q_{ij}x^i y^j$ is defined to be $\max_{(i,j)\in I}iw_x+jw_y$. Let $m,\ell\in\mathbb{Z}^+$ be positive integers, to be determined. Construct a non-zero polynomial $Q(x,y)=\sum_{i,j}q_{ij}x^i y^j$, whose $(1,d)$-weighted degree is at most $m+\ell d $, and $Q(x_i,y_i)=0$, for every $i\in [L]$. The number of coefficients $q_{ij}$ of any such polynomial is at most, $\sum_{j=0}^\ell \sum_{i=0}^{m+(\ell-j)d}1  = (m+1)(\ell+1)+d\ell(\ell+1)/2$. Hence, provided $(m+1)(\ell+1)+d\ell(\ell+1)/2>L$, we have more unknowns (i.e., coefficients $q_{ij}$) than equations, $Q(x_i,y_i)=0$, for $i\in[L]$, and thus, such a $Q(x_i,y_i)$ exists, and moreover, can be found in polynomial time. Now, we look at the following univariate polynomial, $Q(x,f(x))\in\mathbb{F}[x]$. This polynomial has degree, at most $m+\ell d$. Note that, for every $i$ such that $f(x_i)=y_i$, $Q(x_i,f(x_i))=Q(x_i,y_i)=0$. Hence, provided  that, $m,\ell$ are chosen, such that $m+\ell d<t$, it holds that, this polynomial has $t>m+\ell d={\rm deg}Q(x,f(x))$ zeroes, hence, it must be identically zero. Now, viewing $Q(x,y)$ to be $Q_x(y)$, a polynomial in $y$, with coefficients from $\mathbb{F}[x]$, we have that whenever $Q_x(\xi)=0$, it holds that, $(y-\xi)$ divides $Q_x(y)$, hence, for $\xi=f(x)$, we get $y-f(x)\mid Q(x,y)$. Provided that $Q(x,y)$ exists (which will be guaranteed by parameter assumptions) and can be reconstructed in polynomial in $n$ time, it can also be factorized in probabilistic polynomial time \cite{kaltofen1992polynomial}, and $y-f(x)$ will be one of its irreducible factors. For a concrete choice of parameters, see \cite{sudan1996maximum}; or \cite{cai1999hardness}, which also has a brief and different exposition of the aforementioned ideas. We will use this result with $t>\sqrt{2Ld}$, where $d={\rm deg}(\phi)<n^2$. 

Now, we have a randomized procedure, which outputs a certain list $\mathcal{K}$ of at most $3/q$ polynomials, one of which is the correct $\phi(x)=Z_{n-1}(D(x);p_n)$. The idea for the remainder is as follows. We will find a point $x$, at which, all polynomials from the list $\mathcal{K}$ disagree. Towards this goal, define a set $\mathcal{T}$ of triples,
$$
\mathcal{T}=\{(x,f(x),g(x)):f(x)= g(x),x\in\mathbb{Z}_{p_n},f,g\in \mathcal{K}\}.
$$
We now use a double-counting argument. Note that, every pair $(f,g)$ of distinct polynomials from the list $\mathcal{K}$ can agree on at most $n^2-1$ points. Since, the total number of such pairs $(f,g)$ of distinct polynomials from $\mathcal{K}$ is less than $(3/q)^2$, we deduce $|\mathcal{T}|<9n^{2k+2}$. Since $|\mathbb{Z}_{p_n}|>|\mathcal{T}|$, it follows that, there exists a $v$, such that, no triple, whose first coordinate is $v$ belongs to $\mathcal{T}$. Clearly, this point $v$ can be found in polynomial time, since $p_n$ and the  size of the list are polynomial in $n$. Thus, there is at least one point on which all polynomials from the list $\mathcal{K}$ disagree. 
It is possible now to identify $\phi(x)=Z_{n-1}(D(x);p_n)$, by evaluating $Z_{n-1}(D(v);p_n)$, since whp, $\phi(\cdot)\in \mathcal{K}$, and all polynomials from list $\mathcal{K}$ take distinct values at $v$. Provided $\phi(x)$ can be identified, we can compute $Z_n({\bf J,B,C};p_n)$, the original partition function of interest, simply via $C_n'\phi(1)+B_n'\phi(2)$, as mentioned in  the beginning.

Therefore, $Z_n({\bf J,B,C};p_n)$ can be computed, provided that $Z_{n-1}(D(v);p_n)$ can be computed, a reduction from an $n-$spin system, to an $(n-1)-$spin system. Note that, the probability of error in this randomized reduction is upper bounded, via the union bound, by the sum of probabilities that, $\N$, defined in Lemma \ref{whp} is less than $(p_n-2)q/2$, which is of probability at most $\frac{1}{(p_n-2)q^2}$, which is $c/n^2$ for some constant $c>0$, independent of $n$; plus, the probability of failure during the construction of a list of $L$ pairs $(x_i,y_i)_{i=1}^L$ with $t>\sqrt{2Ld}$, which, conditional on the high probability event, $\{\mathcal{N}\geq (p_n-2)q/2\}$, is exponentially small in $n$; and finally, the probability that we encounter an error during generating the list of polynomials through factorization, per Lemma \ref{madhu}, which can again be made exponentially small in $n$. Thus, the overall probability of error for this reduction is $c'/n^2$, for some absolute constant $c'>0$, independent of $n$. Next, select a large $H$ and repeat the same downward reduction protocol $n\to n-1,n-1\to n-2,\cdots,H+1\to H$, such that the total probability of error $\sum_{j=H}^n c'/j^2$ during the entire reduction is less than $1/2$ (note that, the reduction step, $n-1\to n-2$ aims at computing $\phi(v)=Z_{n-1}(D(v);p_n)$, where $v$ is the element of $\mathbb{Z}_{p_n}$ discussed earlier; and each step, we reduce the problem of recovering the associated polynomial to evaluating the partition function of a system with one less number of spins, at a single input point). Once the system has $H$ spins, compute the partition function by hand. 
This procedure yields an algorithm computing $Z_n({\bf J,B,C};p_n)$, the partition function value we wanted to compute in the beginning of the proof of Proposition \ref{prop-main}, with probability greater than  $1/2$. Now, if we repeat this algorithm $R$ times, and take the majority vote (i.e., the number that appeared the majority number of times), the probability of having a wrong answer appearing as majority vote is, by Chernoff bound, exponentially small in $R$. Taking $R$ to be polynomial in $n$, we have that with probability at least $1-e^{-\Omega(n)}$ this procedure correctly computes $Z_n({\bf J,B,C};p_n)$.

We have now established that, provided, there is a polynomial time algorithm $\mathcal{A}$, which exactly computes the partition function on $1/n^k$ fraction of inputs (from $\mathbb{Z}_{p_n}^{n(n-1)/2+2n}$), then there exists a (randomized) polynomial time procedure, for which, for every ${\bf a}\in\mathbb{Z}_{p_n}^{n(n-1)/2+2n}$ (including, in particular, the adversarially-chosen ones), it correctly evaluates $Z_n({\bf a}) \pmod{p_n}$ with probability $1-o(1)$. We now use this procedure to show, how to evaluate $Z_n({\bf a})$ (without the mod operator). We use the Chinese Remainder Theorem, which, for convenience, is stated  below.
\begin{theorem}
Let $p_1,\dots,p_k$ be distinct pairwise coprime positive integers, and $a_1,\dots,a_k$ be integers. Then, there exists a unique integer $m\in\{0,1,\dots,P\}$ where $P=\prod_{\ell=1}^k p_{\ell}$, such that, 
$
m\equiv a_i\pmod{p_i}$, for every $1\leq i\leq k$.
\end{theorem}
In particular, letting $P_i = P/p_i$, $m=\sum_{\ell=1}^k c_i P_i a_i\pmod{P}$ works, where $c_i\equiv P_i^{-1}\pmod{p_i}$. The number $c_i$ can be computed by running Euclidean algorithm: Since ${\rm gcd}(P_i,p)=1$, it follows from B\'ezout's identity that, there exists integers $c_i,b\in\mathbb{Z}$ such that, $c_iP_i+pb=1$, and thus, $c_iP_i\equiv 1\pmod{p}$. Now, we proceed as follows. Fix a positive integer $m$. If we can find a collection $\{p_1,\dots,p_{\ell}\}$ of primes such that the corresponding product $P=\prod_{k=1}^\ell p_k$ exceeds $m$, then we can recover $m$, from $(r_i)_{i=1}^\ell$, where $r_i\in\mathbb{Z}_{p_i}$ is such that $m\equiv r_i\pmod{p_i}$, namely, $r_i$ is the remainder obtained upon dividing $m$ by $p_i$, for each $i$.

For this goal, we now establish a bound, where with high probability, the original partition function is less than this bound. Recall the standard Gaussian tail estimate, $\mathbb{P}(Z>t)= O(\exp(-t^2/2))$. Using this, 
$$
\mathbb{P}(e^{\beta n^{-1/2}J}>t) =O(\exp(-n\log^2(t)/(2\beta^2))),
$$
which, for $t=n$, gives a bound, $o(n^{-2})$. Now, for external field contribution, we have $\mathbb{P}(e^{B}>t)\leq O(\exp(-\log^2(t)/(2\beta^2)))$ (also for $C$), which, for $t=n$, gives $O(n^{-\log n/(2\beta^2)})$, which is, again, $o(n^{-2})$. Hence, with high probability, the $n(n-1)/2+2n$-dimensional vector, ${\bf V}=({\bf J,B,C})$ is such that, $\|{\bf V}\|_\infty \leq n$. Therefore, with high probability, the partition function is at most sum of $2^n$ terms, each of which is a product of at most $n^2$ terms (since, we have $n$ terms for external field, and at most $n^2/2$ terms for spin-spin couplings) each bounded by $2^N n$. This establishes, the partition function is at most $2^n(2^N n)^{n^2}=2^{Nn^2+n^2\log_2 n+O(n)}$.

It now remains to show that, there exists sufficiently many prime numbers of appropriate size, that we can use for Chinese remaindering. 
\begin{lemma}\label{lemma:pnt-enough}
Let $k,\alpha>0$ be a fixed constants, and $N$ satisfies $\Omega(\log n)\leq N\leq n^\alpha$. The number of primes between $9n^{2k+2}$ and $2(2+\alpha+2k)Nn^{2k+2}\log n$ is at least $Nn^{2k+2}$, for all sufficiently large $n$.
\end{lemma}
The proof of this lemma can be found in Section \ref{subsec:pnt}.

Having done this, we will find a sequence of $Nn^{2k+2}$ primes via brute force search in polynomial time, since $N\leq n^\alpha$ for some constant $\alpha$, with $p_j>\Omega(n^{2k+2})$. This will establish, $\prod_j p_j >\Omega((n^{2k+2})^{Nn^{2k+2}})=\Omega(2^{Nn^{2k+2}(2k+2)\log n})$. Since the partition function is at most $2^{Nn^2+n^2 \log_2 n+O(n)}$ and since $N=\Omega(\log_2 n)$, we therefore conclude that the product of primes we have selected is, whp, larger than the partition function itself, and therefore, by running $\mathcal{A}$ with each of these prime basis, and Chinese remaindering, we can compute the partition function exactly. Therefore, the proof of Proposition \ref{prop-main} is complete.
\end{proof}
We now establish that the density of log-Normal distribution is Lipschitz continuous within a finite interval, and will bound the Lipschitz constant, to establish a certain probabilistic coupling.
Recall that $J_{ij}, 1\le i<j\le n$ are i.i.d. standard normal and $\widehat{J}_{ij}=e^{{\beta \over\sqrt{n}} J_{ij}}$. Let $f_{\widehat{J}}$ denote the common density of $\widehat{J}_{ij}$. 
\begin{lemma}\label{lemma:X-nearly-Lipschitz}
For every $0<\delta<\Delta$ satisfying $\log\Delta>\beta^2$ and every $\delta\le t,\tilde t\le\Delta$,
the following bound holds.
\begin{align}\label{eq:multiplicative-bound}
\exp\left(-{2n\log \Delta\over \beta^2\delta}|\tilde t-t|\right) 
\le
{f_{\widehat{J}}(\tilde t)\over f_{\widehat{J}}(t)} 
\le \exp\left({2n\log \Delta\over \beta^2\delta}|\tilde t-t|\right).
\end{align}
\end{lemma}
The proof of this lemma is provided in Section \ref{subsec-lemma-Lip}. Furthermore, letting  $\widehat{B}_i=e^{B_i}$ and $\widehat{C}_i=e^{C_i}$, and denoting the (common) densities by $f_{\widehat{B}}$ and $f_{\widehat{C}}$, we have that the same Lipschitz condition holds also for $f_{\widehat{B}}(t)$ and $f_{\widehat{C}}(t)$, and therefore, the result of Lemma \ref{lemma:X-nearly-Lipschitz} applies also to the exponentiated version of the external field components, see Remark \ref{remark:lipschitz}.

The idea for the remaining part is as follows. We will establish that, the algorithmic inputs (obtained by exponentiating the real-valued inputs and truncating at an appropriate level $N$), are close to uniform distribution (modulo $p_n$), in total variation sense, which will establish the existence of a desired coupling to conclude the proof of Theorem \ref{thm:main}. To that end, we now establish an auxiliary result, showing that the log-Normal distribution is nearly uniform, modulo $p_n$.
\begin{lemma}\label{lemma:Near-Uniformity}
The following bound holds for every $A\in\{\widetilde{J}_{ij}:1\leq i<j\leq n\}\cup\{\widetilde{B}_i:i\in[n]\}\cup\{\widetilde{C}_i:i\in[n]\}$:
\begin{align*}
\max_{0\le \ell\le p_n-1}|\pr(A\equiv\ell \bmod(p_n) )-p_n^{-1}|=O(N^{-1}n^{-5k-4}).
\end{align*}
\end{lemma}
The proof of this lemma is provided in Section \ref{subsec:near-U}. We now return to the proof of Theorem \ref{thm:main}. Using Lemma \ref{lemma:Near-Uniformity}, the total variation distance between any $A\in\{\widetilde{J}_{ij},\widetilde{B}_i,\widetilde{C}_i\}$ and $U\sim{\rm Unif}(\mathbb{Z}_{p_n})$ is at most, $O(p_nN^{-1}n^{-5k-4})$, which, using the trivial inequality $p_n\leq O(Nn^{3k+2})$, is $O(n^{-2k-2})$. We now use the following well-known maximal total variation coupling result.
\begin{theorem}
Let the random variables $X,Y$ have marginal distributions, $\mu$ and $\nu$, and let $d_{TV}(\mu,\nu)$ denotes the total variation distance between $\mu$ and $\nu$. Then, for any coupling (namely, any joint distribution with marginals of $X$ and $Y$ being $\mu$ and $\nu$, respectively) of $X$ and $Y$, it holds that,
$\mathbb{P}(X= Y)\leq 1-d_{TV}(\mu,\nu)$.
Moreover, there is a coupling of $X$ and $Y$, under which, we have the equality $\mathbb{P}(X=Y)=1-d_{TV}(\mu,\nu)$.
\end{theorem}
Using this maximal coupling result, we now observe that, we can couple $A$ (where, $A\in\{\widetilde{J_{ij}},\widetilde{B_i},\widetilde{C_i}\}$) with a  random variable $U$, uniformly distributed on $\mathbb{Z}_{p_n}$, such that
$$
\mathbb{P}(A=U)\geq 1-O(n^{-2k-2}).
$$
Now, let $U_{ij}$, $U^{B}_i$, and $U^{C}_i$ be random variables, uniform over $\mathbb{Z}_{p_n}$, such that,
$$
\mathbb{P}(\widetilde{J}_{ij}\neq U_{ij})\leq O(n^{-2k-2}),\quad\text{and}\quad \mathbb{P}(\widetilde{B}_i\neq U^B_i)\leq O(n^{-2k-2}), \quad\text{and}\quad \mathbb{P}(\widetilde{C}_i\neq U^C_i)\leq O(n^{-2k-2}).
$$
In particular, using union bound, we can couple ${\bf \Xi}=({\bf \widetilde{J},\widetilde{B},\widetilde{C})}$, with a vector, ${\bf U}=({\bf U^{J},U^{B},U^{C}})\in\mathbb{Z}_{p_n}^{n(n-1)/2+2n}$, such that, $\mathbb{P}({\bf \Xi}={\bf U})\geq 1-O(n^{-2k})$. Now, we define several auxiliary probabilistic events. Let $\mathcal{E}_1 = \{Z_n({\bf \Xi};p_n)=Z_n({\bf U};p_n)\}$, $\mathcal{E}_2=\{Z_{\mathcal{A}}({\bf \Xi};p_n)=Z_{\mathcal{A}}({\bf U};p_n)\}$, and $\mathcal{E}_3 = \{Z_{\mathcal{A}}({\bf \Xi})=Z_n({\bf \Xi})\}$. Observe that, due to the coupling, we have $\mathbb{P}(\mathcal{E}_1),\mathbb{P}(\mathcal{E}_2)\geq 1-O(n^{-2k})$. Now, suppose, the statement of the Theorem \ref{thm:main} holds, and that, $\mathbb{P}(\mathcal{E}_3)\geq 1/n^k$. Observe that,  $\mathcal{E}_1\cap \mathcal{E}_2\cap \mathcal{E}_3\subseteq \{Z_n({\bf U};p_n)=Z_{\mathcal{A}}({\bf U};p_n)\}$. Hence, $\mathcal{A}$ satisfies,
\begin{align*}
\mathbb{P}(Z_n({\bf U};p_n)=Z_{\mathcal{A}}({\bf U};p_n))&\geq \mathbb{P}(\mathcal{E}_1\cap \mathcal{E}_2\cap \mathcal{E}_3) \\
& = 1-\mathbb{P}(\mathcal{E}_1^c \cup \mathcal{E}_2^c \cup \mathcal{E}_3^c) \\
&\geq \mathbb{P}(\mathcal{E}_3) -\mathbb{P}(\mathcal{E}_1^c)-\mathbb{P}(\mathcal{E}_2^c) \\
&\geq \frac{1}{n^k}-O(n^{-2k}) \geq \frac{1}{n^{k'}},
\end{align*}
using union bound, where $k'$ obeys: $k<k'<2k$ and $n^{2k'+2}\log n = O(n^{3k+2})$. This contradicts with Proposition \ref{prop-main}, with the probability of success taken to be as $1/n^{k'}$ for this value of $k'$.
\section{Average-Case Hardness under Real-Valued Computational Model}
In this section, we study the problem of exactly computing the partition function associated with the Sherrington-Kirkpatrick model, but this time under the real-valued computation model, as opposed to the finite precision arithmetic model adopted in previous section.

More specifically, we assume that there exists a computational engine, operating over real-valued inputs, and that, each arithmetic operation on real-valued inputs is assumed to be of unit cost. An example of such a computational model is the so-called Blum-Shub-Smale machine \cite{blum1988theory,blum2012complexity}. The techniques employed in the previous section do not extend to real-valued computational model, since it is not clear what the appropriate real-valued analogue of $\Z_p$ is.  
\subsection{Model and the Main Result}
We start by incorporating the cuts induced by the spin assignment $\boldsymbol{\sigma}\in\{-1,1\}^n$, and reduce  the problem to computing a partition function associated with the cuts, in a  manner analogous to the previous setting. Let $\Sigma =\sum_{i<j}J_{ij} = \sum_{\sigma_i\neq \sigma_j}J_{ij}+\sum_{\sigma_i= \sigma_j}J_{ij}$. Note that, $\Sigma$ is independent of the spin assignment $\boldsymbol{\sigma}\in\{-1,1\}^n$, and is computable in polynomial time. Observe also that, 
$
\Sigma - H(\boldsymbol{\sigma}) = 2\sum_{\sigma_i\neq \sigma_j}J_{ij}
$, where $H(\boldsymbol{\sigma})=\sum_{i<j}J_{ij}\sigma_i\sigma_j$. Therefore, 
$$
Z({\bf J}) = \sum_{\boldsymbol{\sigma}\in\{-1,1\}^n}\exp(-H(\boldsymbol{\sigma})) = \sum_{\boldsymbol{\sigma}\in\{-1,1\}^n} \exp(-\Sigma)\exp\left(2\sum_{i<j:\sigma_i\neq \sigma_j}J_{ij}\right).
$$
Letting  $X_{ij}=e^{2J_{ij}}$, we observe  that since $\exp(-\Sigma)$ is a trivially computable constant, it suffices to compute $\widehat{Z}({\bf J})$, where
$$
\widehat{Z}({\bf J})=\sum_{\boldsymbol{\sigma}\in\{-1,1\}^n} \prod_{i<j:\sigma_i\neq \sigma_j}X_{ij}.$$
Note that,  $\widehat{Z}({\bf J})$ involves $X_{ij}$, which are, in turn, derived from $J_{ij}$. Our main result is as  follows:
\begin{theorem}\label{thm:main-2}
Let $\delta\geqslant 1/{\rm poly}(n)>0$ be an arbitrary real number, ${\bf J}=(J_{ij}:1\leqslant i<j\leqslant n)\in \R^{n(n-1)/2}$ with $J_{ij} \distr \mathcal{N}(0,1)$ iid; and $\mathcal{O}$ be an algorithm, such that:
\begin{align*}
\mathbb{P}\left(\mathcal{O}({\bf J})=\widehat{Z}({\bf J})\right)\geqslant \frac34+\delta,
\end{align*}
where the probability is taken with respect to randomness in 
${\bf J}$. Then, $P=\#P$.
\end{theorem}
We here recall one more time that, the input to the algorithm $\mathcal{O}$ is real-valued, and that, the algorithm operates under a real-valued computational engine, e.g. using a Blum-Shub-Smale machine \cite{blum1988theory}.
\subsection{Proof of Theorem \ref{thm:main-2}}
Let ${\bf \mathcal{Q}}=(q_{ij}:1\leqslant i<j\leqslant n)$ be an arbitrary input (of couplings), so that it is $\#P-$hard to compute the associated partition function, $\widehat{Z}({\bf a})$, which, by a slightly abuse the notation, is
$$
\widehat{Z}({\bf a})=\sum_{\boldsymbol{\sigma}\in\{-1,1\}^n}\prod_{i<j:\sigma_i\neq\sigma_j}a_{ij},
$$
with ${\bf a}=(a_{ij}:1\leqslant i<j\leqslant n)$, where $a_{ij}=e^{q_{ij}}$. In particular, $a_{ij}>0$ for any $1\leq i<j\leq n$. Now, let ${\bf J}$ be a vector with iid standard normal components, and let ${\bf X}=(X_{ij}:1\leqslant i<j\leqslant n)$ be a vector, where $X_{ij}=e^{2J_{ij}}$ for every $1\leqslant i<j\leqslant n$. Define ${\bf X}(t)$ via:
\begin{equation}\label{eqn:X-of-t}
{\bf X}(t)=(1-t){\bf X}+t{\bf a},
\end{equation}
where $0\leqslant t\leqslant 1$, and let $f(t)$ be
\begin{equation}\label{eqn:f-of-t}
f(t) = \widehat{Z}({\bf X}(t)) = \sum_{\boldsymbol{\sigma}\in\{-1,1\}^n}\prod_{i<j:\sigma_i\neq \sigma_j}\left((1-t)X_{ij}+ta_{ij}\right).
\end{equation}
Note that, $f(t)$ is a univariate polynomial in $t$, with degree
$$
{\rm deg}(f) = \max_{\boldsymbol{\sigma}\in\{-1,1\}^n} \left|\{(i,j):1\leqslant i<j\leqslant n, \sigma_i\neq\sigma_j\}\right|=\frac{n^2}{2}+o(n),
$$
and $f(1)=\widehat{Z}({\bf a})$. Assuming the existence of an algorithm $\mathcal{O}(\cdot)$ whose probability of success is at least $\frac34+\frac{1}{{\rm poly}(n)}$, we will show the existence of a randomized polynomial time algorithm which, with probability $\frac12+\frac{1}{{\rm poly}(n)}$, recovers the polynomial $f(t)$. In particular repeating this algorithm $R$ times to compute $f(1)$, where $R$ is chosen to be polynomial in $n$; and taking majority vote, the probability that an incorrect value appears more than the half of time is exponentially small by Chernoff bound. Thus, one can compute $\widehat{Z}({\bf a})$ with probability at least $1-\exp(-\Omega(n))$. 
\begin{lemma}\label{lemma:tv-lognormal}
Let ${\bf X}(t)$ be defined as above. Fix any $1\leq i<j\leq n$, and let $X(t)\triangleq X_{ij}(t)$. Then, there exists an absolute constant $\mathcal{C}_{ij}>0$, depending only on $a_{ij}$, such that, $d_{TV}(X(t),X(0))\leqslant \mathcal{C}_{ij}t$ for every $t\in[0,1]$.
\end{lemma}

An informal, information-theoretic way, of seeing the hypothesis of Lemma \ref{lemma:tv-lognormal} is as follows. Using Pinsker's inequality \cite{csiszar2011information, polyanskiy2014lecture}, we have $d_{TV}(X_{ij}(t),X_{ij}(0))\leqslant \kappa\sqrt{D(X_{ij}(t)\|X_{ij}(0))}$, where $D(\cdot \|\cdot)$ is the KL divergence, and $\kappa>0$ is some absolute constant. Next, using the fact that, KL divergence locally looks like  chi-square divergence\footnote{The chi-square divergence can be thought of as a weighted Euclidean $\ell_2$ distance between two probability distributions, defined on the same probability space.} $\chi^2(\cdot \|\cdot)$  (see, e.g. \cite{polyanskiy2014lecture}), one expects for $t$ small, $D(X_{ij}(t)\|X_{ij}(0))\approx O(t^2)$, and thus, $d_{TV}(X_{ij}(t),X_{ij}(0))\approx O(t)$. 

The full proof of this lemma is deferred to  Section \ref{subsec-tv-lognormal}.

We next state a tensorization inequality for the total variation distance.
\begin{lemma}\label{lemma:tv-tensorize}
Let $P_1,\dots,P_\ell$ and $Q_1,\dots,Q_\ell$ be probability measures, defined on the same sample space $\Omega$. Then, 
$$
d_{TV}\left(\otimes_{i=1}^\ell P_i ,\otimes_{i=1}^\ell Q_i\right)\leqslant \sum_{i=1}^\ell d_{TV}(P_i,Q_i).
$$
\end{lemma}
While this lemma is known, we provide a proof in Section \ref{pf:lemma-tv-tensorize} for completeness. 

Using Lemma \ref{lemma:tv-lognormal}, together with the tensorization property above, we deduce $d_{TV}({\bf X}(t),{\bf X}(0))\leqslant \frac{Cn^2t}{2}$, where
$$
C= \sum_{1\leq i<j\leq n}\mathcal{C}_{ij},
$$
the sum of the constants $\mathcal{C}_{ij}$ prescribed by Lemma \ref{lemma:tv-lognormal}. 

Now, let $L=\lceil n^2/\delta\rceil$, and $\epsilon=\frac{\delta}{2Cn^2L}$. For every $k\in[L]$, we will evaluate $\widehat{Z}({\bf X}(\epsilon k))$ via the oracle $\mathcal{O}(\cdot)$, and will use these values to reconstruct $f(t)$, from which, $f(1)=\widehat{Z}({\bf a})$ can be computed. Note that, with this choice of $L$ and $\epsilon$, $d_{TV}({\bf X}(\epsilon k),{\bf X}(0))\leqslant \frac{\delta}{4}$, for every $k\in[L]$. 

Fix an arbitrary $k\in[L]$, and consider a coupling between ${\bf X}(\epsilon k)$ and ${\bf X}(0)$, which maximizes $\mathbb{P}({\bf X}(\epsilon k)={\bf X}(0))$. Note that, in this case, $\mathbb{P}({\bf X}(\epsilon k)={\bf X}(0)) \geqslant 1-d_{TV}({\bf X}(\epsilon k),{\bf X}(0))$. Define the events $\mathcal{E}_1=\{\mathcal{O}({\bf X}(\epsilon k)) = \mathcal{O}({\bf X}(0))\}$, $\mathcal{E}_2=\{\mathcal{O}({\bf X}(0)) = \widehat{Z}({\bf X}(0))\}$, and finally, $\mathcal{E}_3=\{\widehat{Z}({\bf X}(0))=\widehat{Z}({\bf X}(\epsilon k))\}$. Clearly, $\mathbb{P}(\mathcal{E}_1^c),\mathbb{P}(\mathcal{E}_3^c)\leqslant d_{TV}({\bf X}(\epsilon k),{\bf X}(0))$; and $\mathbb{P}(\mathcal{E}_2^c)\leqslant \frac{1}{4}-\delta$, since ${\bf X}(0)=(X_{ij}:1\leqslant i<j\leqslant n)$ with $X_{ij}=\exp(2J_{ij})$ with $J_{ij}\distr \mathcal{N}(0,1)$. Since 
$$
\mathcal{E}_1\cap \mathcal{E}_2\cap \mathcal{E}_3 \subseteq \{\mathcal{O}({\bf X}(\epsilon k))=\widehat{Z}({\bf X}(\epsilon k))\},
$$ 
it follows that, 
$$
\mathbb{P}\left(\mathcal{O}({\bf X}(\epsilon k))=\widehat{Z}({\bf X}(\epsilon k))\right)\geqslant \frac34+\delta-2d_{TV}({\bf X}(\epsilon k),{\bf X}(0))\geqslant  \frac34+\frac{\delta}{2}.
$$

Now, let $I_1,I_2,\dots,I_L$ be Bernoulli random variables, where for each $k\in[L]$, $I_k=1$ if and only if $\mathcal{O}({\bf X}(\epsilon k))  = \widehat{Z}({\bf X}(\epsilon k))$. Clearly, $\mathbb{P}(I_k=1)\geqslant \frac34+\frac{\delta}{2}$.
\begin{lemma}\label{lemma:markov}
Let $X_1,X_2,\dots,X_\ell$ be Bernoulli random variables (not necessarily independent), where there exists $0<q<1$, such that $\mathbb{E}[X_k]\geqslant q$, for every $k\in[\ell]$. Let $0<\epsilon<q$ be arbitrary. Then,
$$
\mathbb{P}\left(\frac{1}{\ell}\sum_{k=1}^\ell X_k>\epsilon\right)\geqslant \frac{q-\epsilon}{1-\epsilon}.
$$
\end{lemma}
\noindent The proof of this lemma is provided in Section \ref{subsec:lemma-markov}. In particular, letting $N=\sum_{k=1}^L I_k$, and using  Lemma \ref{lemma:markov} with $\epsilon=\frac12+\frac{\delta}{2}$ and $q=\frac34+\frac{\delta}{2}$, we deduce
$$
\mathbb{P}\left(N\geqslant \left(\frac12+\frac{\delta}{2}\right)L\right)\geqslant  \frac12+\frac{\delta}{2}.
$$
Let $\mathcal{L}=\{(x_k,y_k):k\in[L]\}$ where $x_k=\epsilon k$, and $y_k=\mathcal{O}({\bf X}(\epsilon k))$. 
The next  result shows, provided $N\geqslant \left(\frac12+\frac{\delta}{2}\right)L$, one can recover $f(t)=\widehat{Z}({\bf X}(t))$ in polynomial time.
\begin{theorem}[{\bf Berlekamp-Welch}] Let $f$ be a univariate  polynomial, with ${\rm deg}(f)=d$ over any field $\mathbb{F}$. Let $\mathcal{L}=\{(x_i,y_i):1\leqslant i\leqslant L\}$ be a list such that, for at least $t$ pairs of the list, where $t>\frac{L+d}{2}$, $y_i=f(x_i)$ holds. Then, there exists an algorithm which recovers $f$, using at most polynomial in $L$ and $d$ many field operations over $\mathbb{F}$.
\end{theorem}
Note that, provided $N\geqslant \left(\frac12+\frac{\delta}{2}\right)L$, the list $\mathcal{L}$ constructed  above will satisfy the requirements of Berlekamp-Welch algorithm, and therefore, the value of $f(1)=\widehat{Z}({\bf a})$ can be computed efficiently, with probability $\frac12+\frac{\delta}{2}$, using at most polynomial in $n$ many arithmetic operations over reals. 

Now we repeat  this process by $R$ times, and take majority vote. The probability  that, a wrong answer will appear as a majority vote, is exponentially small, using Chernoff bound. Taking $R$ to be polynomial in $n$, we deduce this process efficiently computes $\widehat{Z}({\bf a})$ with probability at least $1-\exp(-\Omega(n))$,  which is known to be a $\#P-$hard problem.
\section{Conclusion and Future Work}
In this paper, we have studied the average-case hardness of the algorithmic problem of exactly computing the partition function associated with the Sherrington-Kirkpatrick model of spin glass with Gaussian couplings and random external input. We have established that, unless $P=\#P$, there does  not exists a polynomial time algorithm which exactly computes the partition function on average. We have established our result by combining the approach of Cai et al. \cite{cai1999hardness} for establishing the average-case hardness of computing the permanent of  a (random) matrix, modulo a prime number $p$; with a probabilistic coupling between log-normal inputs and random  uniform inputs over a finite field. To the best of our knowledge, ours is the first such result, pertaining the statistical physics models. We also note that, our approach is not limited to the case of Gaussian inputs: for random variables with sufficiently well-behaved density, for which, one can establish a coupling as in Lemma \ref{lemma:Near-Uniformity} to a prime of appropriate size, our techniques transfer.

Several future research directions are as follows. The proof sketch outlined in this paper, as well as in the previous works \cite{cai1999hardness,feige1992hardness,lipton1989new} do not transfer to the several other fundamental open problems aiming at establishing similar hardness results related to SK model. One such fundamental problem is the problem of exactly computing a ground state, namely, the problem of finding a state $\boldsymbol{\sigma}^*\in\{-1,1\}^n$, such that, $H(\boldsymbol{\sigma}^*)=\max_{\boldsymbol{\sigma}\in\{-1,1\}^n} H(\boldsymbol{\sigma})$. Arora et al. \cite{arora2005non} established that the problem of exactly computing a ground state is NP-hard in the worst case sense. Furthermore, Montanari \cite{montanari2018optimization} recently proposed a message-passing algorithm, which, for a fixed $\epsilon>0$ finds a state $\boldsymbol{\sigma}_*$ such that $H(\boldsymbol{\sigma}_*)\geq (1-\epsilon)\max_{\boldsymbol{\sigma}\in\{-1,1\}^n} H(\boldsymbol{\sigma})$ with high probability, in a time at most $O(n^2)$, assuming a widely-believed structural conjecture in statistical physics. Namely, it is possible to efficiently approximate the ground state of SK model within a multiplicative factor of $1-\epsilon$. The proof techniques of Cai et al. \cite{cai1999hardness}, as well as Lipton's approach \cite{lipton1989new}, do not, however, seem to be useful in addressing the average-case hardness of the algorithmic problem of exactly computing the ground state since the algebraic structure relating the problem into the recovery of a polynomial is lost, when one considers the maximization; and this problem remains open. 

Another fundamental problem, which remains open, is the average-case hardness of the problem of computing the partition function approximately, namely, computing $Z({\bf J,\beta})$ to within a multiplicative factor of $(1\pm \epsilon)$, which has been of interest in the field of approximation algorithms.

Yet another natural question is whether the assumption on the oracle $\mathcal{O}(\cdot)$ in Theorem \ref{thm:main-2}  for the real-valued computational model, that is,
$$
\mathbb{P}\left(\mathcal{O}({\bf J})=\widehat{Z}({\bf J})\right)\geqslant \frac34+\frac{1}{{\rm poly}(n)}
$$
can be weakened e.g., to  $1/2+1/{\rm poly}(n)$ or even to $1/{\rm poly}(n)$, as handled in the finite-precision setting. As we have mentioned previously, our approach for establishing the average-case hardness of the problem of exact computation of the partition function under the finite-precision arithmetic model is in parallel with the line of research dealing with the average-case hardness of computing the permanent over a finite field. A typical result along these lines is obtained under the assumption that there exists an oracle which computes the permanent with a certain probability of success, $q$. The first such result, under the weakest assumption of $q=1-1/3n$, is obtained by Lipton \cite{lipton1989new}. Subsequent research weakened this assumption to $q=3/4+1/{\rm poly}(n)$ by Gemmell et al. \cite{gemmell1991self}, then to $q=1/2+1/{\rm poly}(n)$ by Gemmell and Sudan \cite{gemmell1992highly}; and finally to $q=1/{\rm poly}(n)$, by Cai et al. \cite{cai1999hardness}.

The assumption on the success probability of the oracle that we have adopted in this paper for the real-valued computational model is similar to that of Gemmell et al. \cite{gemmell1991self}, and thus, the most natural question is to ask, whether, at the very least, the technique of Gemmell and  Sudan \cite{gemmell1992highly}
can be applied. We now discuss that this seems to be a challenging task, and show where the extension fails.

The idea of Gemmell and Sudan, essentially,  aims at reconstructing a  certain polynomial (similar to (\ref{eqn:f-of-t})), which is observed through its noisy samples (e.g., similar to the list $\mathcal{L}=\{(x_k,\mathcal{O}(\epsilon k)):k\in[L]\}$, that we have defined  earlier), and is adapted  to our case as follows. Let ${\bf J}=(J_{ij}:1\leqslant i<j\leqslant n)$ and ${\bf J'}=(J_{ij}':1\leqslant i<j\leqslant n)$ be two iid random vectors, each with iid  standard normal components, and let ${\bf X}=(X_{ij}:1\leqslant i<j\leqslant n)$ and ${\bf X'}=(X_{ij}':1\leqslant i<j\leqslant n)$, where $X_{ij}=e^{2J_{ij}},X_{ij}'=e^{2J_{ij}'}$, for $1\leqslant i<j\leqslant n$. Define:
\begin{equation}\label{eqn:modified-X-of-t}
{\bf X}(t)=t(1-t){\bf X}+(1-t){\bf X'}+t^2{\bf a},
\end{equation}
where ${\bf a}$ is a worst-case input. Note that, the sampling set $\{{\bf X}(t):t\in[0,1]\}$ is defined more carefully, by incorporating an extra randomness via ${\bf X'}$  (cf. equation (\ref{eqn:X-of-t})). The purpose  of this extra randomness in Gemmell and Sudan's work was to bring  pairwise independence,  that is to ensure the independence of ${\bf X}(t)$ and ${\bf X}(t')$ for $t\neq t'$, in order to be able to use a tighter concentration inequality (namely, Chebyshev's inequality) as a replacement of our Lemma \ref{lemma:markov} while obtaining a high probability guarantee on the constructed list. In their work, this is successful: ${\bf X}$ and ${\bf X'}$ consist of iid samples,  drawn independently from uniform distribution over a finite field $\mathbb{F}_p$, in which case, it is not hard to show, ${\bf X}(t)$  and ${\bf X}(t')$ are always independent for $t\neq t'$. For us, however, this is no longer true: ${\bf X}$ and ${\bf X'}$ both consist of iid log-normal components, which breaks down uniformity and independence. 

We leave the following problem open for future work: Let ${\bf J}=(J_{ij}:1\leqslant i<j\leqslant n)\in\R^{n(n-1)/2}$ be a random vector with $J_{ij}\distr \mathcal{N}(0,1)$, iid. Suppose that, there is an algorithm $\mathcal{A}(\cdot)$, such that 
$$
\mathbb{P}(\widehat{Z}({\bf J}) =\mathcal{A}({\bf J})) \geqslant \frac12+\frac{1}{{\rm poly(n)}},
$$
and that, the algorithm operates over real-valued inputs. Then $P=\#P$. An even more challenging variant of this problem is to establish the same result, under a weaker assumption on the success probability of the algorithm:
$$
\mathbb{P}(\widehat{Z}({\bf J}) =\mathcal{A}({\bf J})) \geqslant \frac{1}{{\rm poly(n)}}.
$$

As we have noted, our approach is not limited to the Gaussian inputs, so long as the  distributions involved are well-behaved. The current method, however, does not address the case of couplings with iid Rademacher inputs, and the average-case hardness of the exact computation of partition function with iid Rademacher couplings remains open. It is not surprising though in light of the fact that the average-case hardness of the problem of computing the permanent of a matrix with 0/1 entries remains open, as well.

\section{Appendix : Proofs of the Technical Lemmas}
\subsection{Proof of Lemma \ref{lemma:X-nearly-Lipschitz}}\label{subsec-lemma-Lip}
\begin{proof}
The density of $\widehat{J}_{ij}$ is given by
\begin{align}
f_{\widehat{J}}(t)&={d\over dt}\pr\left(e^{{\beta \over\sqrt{n}} J}\le t\right) \notag \\
&={d\over dt}\pr\left(J\le \sqrt{n}{\log t \over \beta}\right) \notag \\
&={\sqrt{n}\over \sqrt{2\pi}\beta t}e^{-n{\log^2 t\over 2\beta^2}}   \notag. 
\end{align}
Here $J$ denotes the  standard normal random variable. 
It is easy to see that 
\begin{align}
f_{\widehat{J}}(t)=O\left(\sqrt{n}t\right), \label{eq:f-density1}
\end{align}
as $t\downarrow 0$ since $e^{\log^2 x}$ diverges   faster than $x^c$ for every constant $c$ as $x\to\infty$. Also
\begin{align}
f_{\widehat{J}}(t)=O\left({\sqrt{n}\over t^2}\right), \label{eq:f-density2}
\end{align}
as $t\to\infty$. Both bounds are very crude of course, but suffice for our purposes.

We have for every $t, \tilde t>0$
\begin{align*}
|\log f_{\widehat{J}}(\tilde t)-\log f_{\widehat{J}}(t)|\le |\log(\tilde t)-\log t|+{n\over 2\beta^2}|\log^2(\tilde t)-\log^2(t)|.
\end{align*}
Now since $|{d\log t \over dt}|=1/t\le 1/\delta$ for $t\ge \delta$, we obtain that in the range $0<\delta\le t, \tilde t\le \Delta$
\begin{align*}
&|\log f_{\widehat{J}}(\tilde t)-\log f_{\widehat{J}}(t)|\le (1/\delta)|\tilde t-t|, \\
& |\log^2(\tilde t)-\log^2(t)|\le {2\log \Delta\over \delta}|\tilde t-t|.
\end{align*}
Applying these bounds, exponentiating, and using the assumption on the lower bound on $\log\Delta$ and $n>\beta^2$, we obtain 
\begin{align}\label{eq:multiplicative-bound}
\exp\left(-{2n\log \Delta\over \beta^2\delta}|\tilde t-t|\right) 
\le
{f_{\widehat{J}}(\tilde t)\over f_{\widehat{J}}(t)} 
\le \exp\left({2n\log \Delta\over \beta^2\delta}|\tilde t-t|\right).
\end{align}
\qedhere
\end{proof}
\begin{remark}\label{remark:lipschitz}
Let  $\widehat{B}_i=e^{B_i}$ and $\widehat{C}_i=e^{C_i}$; and denote the (common) densities by $f_{\widehat{B}}$ and $f_{\widehat{C}}$.
$
f_{\widehat{B}}(t)=f_{\widehat{C}}(t)=\frac{1}{\sqrt{2\pi}t}\exp\left(-\frac{\log^2 t}{2}\right)
$,
and therefore, as $t\downarrow 0$,
$f_{\widehat{B}}(t)=O(t)= O(\sqrt{n}t)$, 
and furthermore, as $t\to \infty$, $f_{\widehat{B}}(t)=O(1/t^2)=O(\sqrt{n}/t^2)$. Similarly, the same Lipschitz condition holds, also for $f_{\widehat{B}}(t)$ and $f_{\widehat{C}}(t)$, and therefore, the result of Lemma \ref{lemma:X-nearly-Lipschitz} applies also to the exponentiated version of the external field components. Note also that, we still have the same asymptotic behaviour, even if the external field components $B_i$ and $C_i$ have a constant variance, different than $1$.
\end{remark}
\subsection{Proof of Lemma \ref{lemma:self-recursion}}\label{subsec:self-recursion}
\begin{proof}
We begin by deriving a downward self recursion formula for $I_n(\boldsymbol{\sigma})=|\{(i,j):1\leq i<j\leq n,\sigma_i\neq \sigma_j\}|$. Note that, for a given spin configuration $\boldsymbol{\sigma}\in\{-1,1\}^n$ if $\sigma_n=+1$, then $I_n(\boldsymbol{\sigma})=I_{n-1}(\boldsymbol{\sigma}) + |\{i:\sigma_i=-1,1\leq i\leq n-1\}|$, where we take the projection of $\boldsymbol{\sigma}$ onto its first $(n-1)$ coordinates. Similarly, if $\sigma_n=-1$, then $I_n(\boldsymbol{\sigma})=I_{n-1}(\boldsymbol{\sigma})+|\{i:\sigma_i=+1,1\leq i\leq n-1\}|$. For a given spin configuration $\boldsymbol{\sigma}$, and dimension $n-1$, recalling the definition of $f(n,\boldsymbol{\sigma})$ in (\ref{eq:f-of-n-sigma}), we observe that for $\sigma_n=+1$
$$
f(n,\boldsymbol{\sigma})-f(n-1,\boldsymbol{\sigma})=(n-2)-|\{i:\sigma_i=-1,1\leq i\leq n-1\}|,
$$
and similarly, for $\sigma_n=-1$, 
$$
f(n,\boldsymbol{\sigma})-f(n-1,\boldsymbol{\sigma})=(n-2)-|\{i:\sigma_i=+1,1\leq i\leq n-1\}|.
$$
Now, observe that, using the relation between $f(n,\boldsymbol{\sigma})$ and $f(n-1,\boldsymbol{\sigma})$ with respect to polarity of $\sigma_n$, we have:
\begin{align*}
Z_n({\bf J,B,C};p_n)&=C_n 2^{(n-2)N}\sum_{\substack{\boldsymbol{\sigma}\in\{-1,1\}^{n-1} \\ \sigma_n=+1}} 2^{Nf(n-1,\boldsymbol{\sigma})} \left(\prod_{\substack{1\leq i\leq n-1 \\\sigma_i=-}}2^{-N}B_iJ_{in}\right)\left(\prod_{\substack{1\leq i\leq n-1 \\\sigma_i=+}}C_i\right)\left(\prod_{\substack{1\leq i<j\leq n-1\\\sigma_i\neq \sigma_j}}J_{ij}\right)\\
&+B_n 2^{(n-2)N}\sum_{\substack{\boldsymbol{\sigma}\in\{-1,1\}^{n-1} \\ \sigma_n=-1}}2^{Nf(n-1,\boldsymbol{\sigma})} \left(\prod_{\substack{1\leq i\leq n-1 \\\sigma_i=-}}B_i\right)\left(\prod_{\substack{1\leq i\leq n-1 \\\sigma_i=+}}2^{-N}C_iJ_{in}\right)\left(\prod_{\substack{1\leq i<j\leq n-1\\\sigma_i\neq \sigma_j}}J_{ij}\right)\\
&=C_n' Z_{n-1}({\bf J',B^+,C^+};p_n) + B_n' Z_{n-1}({\bf J',B^-,C^-};p_n).
\end{align*}
\qedhere
\end{proof}
\subsection{Proof of Lemma \ref{pairwise-Indep}}\label{subsec:lemma-pairwise-indep}
\begin{proof}
Fix an $1\leq i\leq T$, and let $\xi_i$ denote the $i^{th}$ component of an arbitrary vector $\xi$. Note that, the event $\{D(x_1)=y_1,D(x_2)=y_2\}$ implies:
\begin{align*}
(y_1)_i&=(2-x_1)(v_1)_i +(x_1-1)(v_2)_i + (x_1-1)(x_1-2)(K_i+x_1M_i) \\
(y_2)_i&=(2-x_2)(v_1)_i +(x_2-1)(v_2)_i + (x_2-1)(x_2-2)(K_i+x_2M_i).
\end{align*}
Since this is a pair of equations with two unknowns (namely, $K_i$ and $M_i$), it has a unique solution, which holds with probability $1/p_n^2$ (note that, $x_1,x_2\notin\{1,2\}$, hence for $i=1,2$, $(x_i-1)(x_i-2)$ terms are not zero, and thus their modulo $p_n$ inverse exists). Finally, using independence across $i\in\{1,2,\dots,T\}$, we get $\mathbb{P}(D(x_1)=y_1,D(x_2)=y_2)=1/p_n^{2T}$. For $\mathbb{P}(D(x_1)=y_1)$, it is not hard to show  by conditioning that, this event has probability $1/p_n^T$.
\end{proof}
\subsection{Proof of Lemma \ref{whp}}\label{subsec:lemma-whp}
\begin{proof}
Let $\N_x\in\{0,1\}$, $x=3,4,\dots,p_n$, be  random variables, where $\N_x=1$ iff $\mathcal{A}(D(x))=\phi(x)=Z_{n-1}(D(x);p_n)$. Namely, $\N_x\sim{\rm Ber}(q)$. Note that, $\N=\sum_{x=3}^{p_n}\N_x$. Let $Z=\N/(p_n-2)$. We have $\mathbb{E}[Z]=q$. Hence,
\begin{align*}
\mathbb{P}(\N<(p_n-2)q/2) =\mathbb{P}\left(\frac{\sum_{x=3}^{p_n} \N_x }{p_n-2}<q/2\right) &= \mathbb{P}(Z-\mathbb{E}[Z]<-q/2)\\
&\leq \mathbb{P}(|Z-\mathbb{E}[Z]|>q/2) \\
&\leq \frac{{\rm Var}(Z)}{(q/2)^2} \leq \frac{1}{(p_n-2)q^2},
\end{align*}
by Chebyshev's inequality, and the trivial inequality, $4q-4q^2\leq 1$. Note that, since we only have pairwise independence as opposed to iid, a Chernoff-type bound do not apply.
\end{proof}
\subsection{Proof of Lemma \ref{poly-in-n-F}}\label{subsec-poly-in-n-F}
\begin{proof}
Assume the contrary, and take a subset $\mathcal{F'}\subseteq\mathcal{F}$ with $|\mathcal{F'}|=\lceil 3/q\rceil$. Let, $G_{\mathcal{S}}(f)=\{i:(i,f(i))\in G(f)\cap \mathcal{S}\}$. Note that, $\bigcup_{f\in \mathcal{F}}G_{\mathcal{S}}(f)\subseteq \{3,4,\dots,p_n\}$, and furthermore, for any distinct $f,f'\in\mathcal{F'}$, it holds that, $|G_{\mathcal{S}}(f)\cap G_{\mathcal{S}}(f')|\leq n^2-1$. Indeed, if not, define $\widehat{f}=f-f'$, and observe that ${\rm deg}(\widehat{f})\leq n^2-1$. If $|G_{\mathcal{S}}(f)\cap G_{\mathcal{S}}(f')|\geq n^2$, then, on at least $n^2$ values of $i$, $f(i)=f'(i)$, and thus, $\widehat{f}(i)=0$, yielding that $\widehat{f}$ has at least $n^2$ distinct zeroes (modulo $p_n$), a contradiction to the degree of $\widehat{f}$. Now, using inclusion-exclusion principle, 
\begin{align*}
p_n-2 \geq \left|\bigcup_{f\in \mathcal{F'}} G_{\mathcal{S}}(f)\right|&\geq \sum_{f\in \mathcal{F'}}|G_{\mathcal{S}}(f)| - \sum_{f,f'\in\mathcal{F'},f\neq f'}|G_{\mathcal{S}}(f)\cap G_{\mathcal{S}}(f')| \\
&\geq \lceil \frac{3}{q}\rceil\frac{(p_n-2)q}{2}- \frac{1}{2}\lceil\frac3q\rceil (\lceil \frac3q\rceil -1)(n^2-1)\\
&=\frac12\lceil\frac3q\rceil \left((p_n-2)q -(\lceil 3/q\rceil -1)(n^2-1) \right) \\
&\geq (p_n-2)+\frac{p_n-2}{2} - \frac{3}{2q}(\lceil 3/q\rceil -1)(n^2-1).
\end{align*}
However, contradicting with this inequality, we claim that in fact $p_n-2 >\frac{3}{q}(\lceil 3/q\rceil -1)(n^2-1)$. Since $\lceil 3/q\rceil < 3/q+1$, it is sufficient to show that, $p_n-2>\frac{9}{q^2}(n^2-1)$. Since $q\geq 1/n^k$, we have $\frac{9}{q^2}(n^2-1)\leq 9n^{2k}(n^2-1)=9n^{2k+2}-9n^{2k}<p_n-2$, for $n$ large (for any $k$). Hence, we arrive at a contradiction.
\end{proof} 
\subsection{Proof of Lemma \ref{lemma:list-decod}}\label{subsec:list-decod}
\begin{proof}
We condition  on the high probability event, $\{\N\geq (p_n-2)q/2\}$, where $\N$ is the random variable defined in Lemma \ref{whp}. We divide the construction, into two cases, depending on the magnitude of $p_n$ that we are working at. 

First, suppose $9n^{2k+2}\leq p_n\leq 161n^{3k+2}$. Apply $\mathcal{A}$ on $D(x)$, for every $x=3,4,\dots,p_n$ (which, due to magnitude constraint on $p_n$, takes at most polynomial in $n$ many operations). By Lemma \ref{whp}, with probability at least $1-\frac{1}{(p_n-2)q^2}$,
    $\mathcal{A}(D(x))=\phi(x)=Z_{n-1}(D(x);p_n)$ for at least $\frac{(p_n-2)q}{2}$ points. Now, since $q\geq 1/n^k$, we have a list $(x_i,y_i)_{i=1}^L$ (where $L=p_n-2$ and $y_i=\mathcal{A}(D(x))$), and there is a polynomial $f$ of degree $d$ less than  $n^2$ (namely, $\phi(x)=Z_{n-1}(D(x);p_n)$), such that, the graph of $f$  intersects the list at at least $t=\frac{p_n-2}{2n^k}$ points. As $p_n\geq 9n^{2k+2}$, it holds that $t>\sqrt{2Ld}$. Clearly, for all such pairs, the first coordinates are all distinct.

Next, suppose $p\geq 161n^{3k+2}$. In this case, it is not clear, whether running the algorithm on $\{D(x):x=3,4,\dots,p_n\}$ takes polynomial in $n$ many calls to $\mathcal{A}$. To handle this issue, we apply the following resampling procedure (where the choice of numbers is to make sure the argument works). Select $L=40n^{2k+2}$ numbers $x_1,x_2,\dots,x_L$, uniformly and independently from $\{3,4,\dots,p_n\}$. Our goal is to find a lower bound on the number of $x_i$'s, for which with high probability we have at least a certain number of distinct $x_i$'s, on which $\mathcal{A}$ run correctly. We claim that, with high probability, we will end up with at least $9n^{k+2}$ distinct $x_i$'s on which $\mathcal{A}(D(x_i))=Z_{n-1}(D(x_i);p_n)$. We argue as follows. Define a collection $\{E_j:1\leq j\leq L\}$ of events,
    $$
    E_j=\{x_j \neq x_i, \text{ for }i\leq j-1,\, \mathcal{A}(D(x_j))=\phi(x_j)=Z_{n-1}(D(x_j);p_n)\}.
    $$
    Namely, $E_j$ is the event that, $(x_j,y_j)$ is a 'nice' sample, in the sense that, $x_j$ is distinct from all preceding $x_i$'s, and $y_j=\phi(x_j)=Z_{n-1}(D(x_j);p_n)$. Now, we can change the perspective slightly, and imagine that, $(x_j,y_j)$ is samples from a set, where $x_j\in\{3,4,\dots,p_n\}$, and $y_j = \mathcal{A}(D(x_j))$.  Recall that, among the set $\{D(x):x=3,\dots,p_n\}$, the algorithm computes the partition function on at least $\frac{(p_n-2)q}{2}\geq \frac{p_n-2}{2n^k}$ locations (conditional on the high probability event $\{\mathcal{N}\geq (p_n-2)q/2\}$ of Lemma \ref{whp}). Note that, 
    $$
    \mathbb{P}(E_j)\geq \frac{\frac{p_n-2}{2n^k}-L}{p_n-2}=\frac{1}{2n^k}-\frac{40n^{2k+2}}{161n^{3k+2}-2}\geq \frac{1}{4n^k},
    $$
    since, the worst case for $E_j$ is that, all preceding chosen entries are distinct, leaving less number of choices for $x_j$, and we repeat the procedure $L$ times. With this, we now claim that with high probability, at least $9n^{k+2}$ of events $(E_j)_{j=1}^L$ occur. To see this,  we note that, the event of interest ($9n^{k+2}$ of events $(E_j:j\in L)$ occur), is stochastically dominated by the event that, a binomial random variable ${\rm Bin}(L,1/4n^k)$, whose expectation is $L/4n^k=10n^{k+2}$ is at least $L=9n^{k+2}$, which,  by a standard Chernoff bound, is exponentially small. At the end , we have a list of $L=40n^{2k+2}$ pairs, $(x_i,y_i)_{i=1}^L$, on which we have at least $t\geq 9n^{k+2}$ correct evaluations (whp), where $t\geq 9n^{k+2}>\sqrt{2Ld}$ with $d=n^2$.
\qedhere
\end{proof}

\subsection{Proof of Lemma \ref{lemma:pnt-enough}}\label{subsec:pnt}
\begin{proof}
Suppose, this is false, and the number of primes between $9n^{2k+2}$ and $2(2+\alpha+2k)Nn^{2k+2}\log n$ is at most $Nn^{2k+2}$, for all large $n$. Recall that, prime number theorem (PNT) states,
$$
\lim_{m\to\infty}\frac{\pi(m)}{m/\log m}=1,
$$
where $\pi(m)=\sum_{p\leq m, p\text{ prime}}1$ is the prime counting function.  Now we have, for $m\triangleq 2(2+\alpha+2k)Nn^{2k+2}\log n$, $\pi(m)\leq Nn^{2k+2}+9n^{2k+2}=Nn^{2k+2}(1+o(1))$. Now, using $N\leq n^\alpha$, we have, $\log m\leq (2+\alpha+2k+o(1))\log n$, and therefore,
$$
\frac{m}{\log m} \geq \frac{2(2+\alpha+2k)Nn^{2k+2}\log n}{(2+\alpha+2k+o(1))\log n}=2(1-o(1))Nn^{2k+2},
$$
and since $\pi(m)\leq Nn^{2k+2}(1+o(1))$, we get a contradiction with PNT, for $n$ large enough.
\end{proof}
\subsection{Proof of Lemma \ref{lemma:Near-Uniformity}}\label{subsec:near-U}
\begin{proof}
We have for every $\ell\in [0,p_n-1]$
\begin{align*}
\pr(A=\ell \bmod(p_n) )=\sum_{m\in \Z}\int_{mp_n+\ell\over 2^N}^{mp_n+\ell+1\over 2^N}f_X\left( t\right)dt.
\end{align*}
We now let,
$$
M^*(n)=\frac{n^{5k+9/2}N2^N}{p_n}  \quad\text{and}\quad M_*(n)=\frac{2^N}{Nn^{5k/2+3}p_n}.
$$
Note the following bound on the size of $p_n=o(Nn^{3k+2})$, due to Lemma \ref{lemma:pnt-enough}.
We now consider separately the case $m\in [M_*(n),M^*(n)-1]$ and 
$m\notin [M_*(n),M^*(n)-1]$. For $m\in [M_*(n),M^*(n)-1]$ applying  Lemma~\ref{lemma:X-nearly-Lipschitz} with
\begin{align*}
\delta &={M_*(N)p_n\over 2^N}  \\
\Delta &={M^*(N)p_n\over 2^N}, 
\end{align*}
 we have for very $t$ and $\tilde t$ such that
\begin{align*}
&2^N\tilde t\in [mp_n+\ell,mp_n+\ell+1] \\
&2^Nt\in [mp_n,mp_n+1]
\end{align*}

\begin{align*}
{f_X \left(\tilde t\right)\over f_X \left( t\right)} 
\le \exp\left({2n 2^N \log\left({M^*(n)p_n\over  2^N}\right)    \over \beta^2 M_*(n)p_n} |\tilde t-t|\right).
\end{align*}
Since $|\tilde t-t|\le p_n/2^N$, we obtain
\begin{align*}
{f_X \left(\tilde t\right)\over f_X \left( t\right)} 
\le \exp\left({2n \log\left({M^*(n)p_n\over 2^N}\right)    \over \beta^2 M_*(n)} \right).
\end{align*}
Applying the value of  and $M^*(n)$  we have $\log\left({M^*(n)p_n\over 2^N}\right)=O(\log n)$.
Given an upper bound $p_n=O(Nn^{3k+2})$, we have that the exponent is 
\begin{align*}
O\left({n \log n \over M_*(n)}\right)&=O\left({n^{11k/2+6} N^2 \over 2^N}\right).
\end{align*}
It is easy to check that 
\begin{align*}
O\left({n^{11k/2+6} N^2 \over 2^N}\right)=O\left(N^{-1}n^{-5k-4}\right).
\end{align*}
Indeed, this holds, provided that we ensure:
$$
2^N >\mathcal{C}N^3 n^{21k/2+10} \iff N>3\log N + (21k/2+10)\log n+\log \mathcal{C},
$$
for some constant $\mathcal{C}$. Since $N\leq n^\alpha$, by assumption, it follows that, $3\log N \leq 3\alpha \log n$, hence, it boils down verifying,
$$
N>(3\alpha+21k/2+10)\log n+\log \mathcal{C},
$$
which is due to the hypothesis on $N$ stating $N\geq C(\alpha,k)\log n$ with $C(\alpha,k)=3\alpha+21k/2+10+\epsilon$, for some $\epsilon>0$. 
Thus  the term above is 
\begin{align*}
O\left(N^{-1}n^{-5k-4}\right).
\end{align*}
We obtain a bound
\begin{align*}
\exp\left(O\left(N^{-1}n^{-5k-4}\right)\right)=1+O\left(N^{-1}n^{-5k-4}\right).
\end{align*}
Similarly, we obtain for the same range of $t,\tilde t$
\begin{align*}
{f_X \left(\tilde t\right)\over f_X \left( t\right)}\ge 1-O\left(N^{-1}n^{-5k-4}\right).
\end{align*}
Thus
\begin{align*}
&|\sum_{M_*(n)\le m\le M^*(n)}
\left(\int_{mp_n+\ell\over 2^N}^{mp_n+\ell+1\over 2^N}f_X(t)dt
-\int_{mp_n\over 2^N}^{mp_n+1\over 2^N}f_X(t) dt\right)| \\
&|\sum_{M_*(n)\le m\le M^*(n)}
\int_{mp_n\over 2^N}^{mp_n+1\over 2^N}\left(f_X\left(t+{\ell\over 2^N}\right)-f_X(t)\right) dt| \\
&\le 
O\left(N^{-1}n^{-5k-4}\right)\sum_{M_*(n)\le m\le M^*(n)}\int_{mp_n\over 2^N}^{mp_n+1\over 2^N}f_X(t)dt \\
&=O\left(N^{-1}n^{-5k-4}\right),
\end{align*}
as the sum above is at most the integral of the density function, and thus at most $1$.

We now consider the case $m\le M_*(n)$. We have applying (\ref{eq:f-density1})
\begin{align*}
\int_0^{M_*(n)p_n\over 2^N}f_X(t)dt=O\left(\left({M_*(n)p_n\over 2^N}\right)^2\sqrt{n} \right)
\end{align*}
which applying the value of $M_*(n)$ is $O\left(N^{-2}n^{-5k-6+1/2}\right)=O\left(N^{-1}n^{-5k-4}\right)$.

Finally, suppose $m\ge M^*(n)$. Applying (\ref{eq:f-density2})
\begin{align*}
\int_{t\ge {M^*(n)p_n\over 2^N}}f_X(t)dt=O\left({\sqrt{n}\over {M^*(n)p_n\over 2^N}}\right)=O(N^{-1}n^{-5k-4}).
\end{align*}

We conclude that 
\begin{align*}
\max_{0\le \ell\le p_n-1}|\pr(A_{ij}=\ell \bmod(p_n) )-\pr(A_{ij}=0 \bmod(p_n) )|=O(N^{-1}n^{-5k-4}).
\end{align*}
Thus
\begin{align*}
&\pr(A_{ij}=\ell \bmod(p_n))-p_n^{-1} \\
&={p_n \pr(A_{ij}=\ell \bmod(p_n) )-\sum_{\ell} \pr(A_{ij}=\ell \bmod(p_n) ) \over p_n} \\
&\le {p_n \left(\pr(A_{ij}=0 \bmod(p_n) )+O(N^{-1}n^{-5k-4})\right)-p_n\left(\pr(A_{ij}=0 \bmod(p_n) )-O(N^{-1}n^{-5k-4})\right)\over p_n} \\
&=O(N^{-1}n^{-5k-4}),
\end{align*}
completing the proof of the lemma. A lower bound $O(N^{-1}n^{-5k-4})$ is shown similarly.\qedhere
\end{proof}

\subsection{Proof of Lemma \ref{lemma:tv-lognormal}}\label{subsec-tv-lognormal}
\begin{proof}
Let $X(\lambda) = (1-\lambda)e^J+\lambda a$, with $a>0$, and $J\distr \N(0,4)$. Note that, the density of $J$ is $f_J(t) = \frac{1}{\sqrt{8\pi}}\exp(-t^2/8)$, for every $t\in\R$. Fix a $\lambda_0\in(0,1)$. Note that since the total variation distance is upper bounded by one, it suffices to establish that there exists a constant $\mathcal{C}_{\lambda_0}$, such that
$$
d_{TV}(X(\lambda),X(0)) \leq \mathcal{C}_{\lambda_0} \lambda, \quad \quad \forall \lambda\in[0,\lambda_0].
$$
We begin with a calculation of the density of $X(\lambda)$. Note that, the density of $X(\lambda)$ is supported on $[\lambda a,\infty)$. Fix a $t\geq \lambda a$. Observe that, $\mathbb{P}\left(X(\lambda)\leq t\right) = \mathbb{P}\left(J\leq \log\left(\frac{t-\lambda a}{1-\lambda}\right)\right)$, and thus, differentiation with respect to $t$ yield the density of $X(\lambda)$ to be:
$$
 f_{X(\lambda)}(t) = \frac{1}{\sqrt{8\pi}(t-\lambda a)} \exp\left(-\frac18 \log\left(\frac{t-\lambda a}{1-\lambda}\right)^2\right),\quad\quad \forall t\geq \lambda a.
$$
Now let 
$$
f_X(t) = \frac{1}{\sqrt{8\pi}t}\exp\left(-\frac18\log(t)^2\right),
$$
be the density of the log-normal $e^{J}$, with $J\distr \N(0,4)$. Observe that, wherever it is defined,
$$
f_{X(\lambda)}(t) = \frac{1}{1-\lambda}f_X\left(\frac{t-\lambda a}{1-\lambda}\right).
$$
Recall next the definition of the TV distance, for two continuous random variables $Y,Z$ with densities $f_Y$ and $f_Z$, respectively: $d_{TV}(Y,Z) = \frac12\int |f_Y(t)-f_Z(t)|\; dt$. In particular, we need to control the following quantity:
\begin{align}
    d_{TV}(X(\lambda),X(0)) &=\frac12\int_{-\infty}^{\infty}|f_{X(\lambda)}(t) - f_X(t)|\; dt \\
    &= \frac12\int_0^{\lambda a}f_X(t)\; dt +\frac12 \int_{\lambda a }^{\infty}|f_{X(\lambda)}(t) - f_X(t)|\; dt \\
    &\leq\frac12\mathcal{M}_1 \lambda a+\frac12\int_{\lambda a}^\infty |f_{X(\lambda)}(t) - f_X(t)|\; dt \label{eq:temp}
\end{align}
where $\mathcal{M}_1 = \sup_{t\in\R}f_X(t)$, which is easily found to be finite. 
With this, we now focus on bounding the second term:
\begin{align*}
|f_{X(\lambda)}(t) - f_X(t)| &= \left|\frac{1}{1-\lambda}f_X\left(\frac{t-\lambda a}{1-\lambda}\right) - f_X(t) \right|\\
& \leq \left|\frac{1}{1-\lambda}f_X\left(\frac{t-\lambda a}{1-\lambda}\right) - \frac{1}{1-\lambda}f_X(t)\right| + \left|\frac{1}{1-\lambda} f_X(t)-f_X(t)\right| \\
&\leq \frac{1}{1-\lambda_0} \left|f_X\left(\frac{t-\lambda a}{1-\lambda}\right) - f_X(t)\right|+ \frac{\lambda}{1-\lambda_0}f_X(t),
\end{align*}
where the first inequality uses the triangle inequality, and the second inequality uses the fact that $\lambda\leq \lambda_0<1$. We then have:
\begin{align}
    \int_{\lambda a}^{\infty}|f_{X(\lambda)}(t) - f_X(t)| \; dt &\leq \frac{1}{1-\lambda_0}\int_{\lambda a}^{\infty}\left|f_X\left(\frac{t-\lambda a}{1-\lambda}\right) - f_X(t)\right|\;dt + \frac{\lambda}{1-\lambda_0}\int_{\lambda a}^\infty f_X(t)\; dt \\
    &\leq \frac{1}{1-\lambda_0}\int_{\lambda a}^{\infty}\left|f_X\left(\frac{t-\lambda a}{1-\lambda}\right) - f_X(t)\right|\; dt + \frac{\lambda}{1-\lambda_0}, \label{eq:temp2}
\end{align}
using the fact that $f_X(t)$ is a legitimate density, and thus, $f_X(t)\geq 0$ and $\int_0^\infty f_X(t)=1$. Combining everything we have thus far, in particular, Equations (\ref{eq:temp}) and (\ref{eq:temp2});  we arrive at: 
\begin{equation}\label{eq:tv-X-lambda-X-zero}
    d_{TV}(X(\lambda),X(0))\leq \lambda\left(\frac{1}{2(1-\lambda_0)}+\frac{a}{2}\mathcal{M}_1\right) + \frac{1}{2(1-\lambda_0)}\int_{\lambda a}^{\infty}\left|f_X\left(\frac{t-\lambda a}{1-\lambda}\right) - f_X(t)\right|\; dt,
\end{equation}
where $\mathcal{M}_1 = \sup_{t\in \R}f_X(t)$, is the maximum value of the log-normal density, which is a finite absolute constant. The remaining task is to bound the integral in Equation (\ref{eq:tv-X-lambda-X-zero}). Now, let
$$
I_{t,\lambda} = \left(\min\left(\frac{t-\lambda a}{1-\lambda},t\right),\max\left(\frac{t-\lambda a}{1-\lambda},t\right)\right).
$$
We now make the following observation: 
$$
\frac{t-\lambda a}{1-\lambda}\geq t \iff t-\lambda a\geq t-\lambda t\iff t\geq a.
$$
Namely, we have that for $t\geq a$:
\begin{equation}\label{eq:I-t-lambda}
I_{t,\lambda} = \left(t,\frac{t-\lambda a}{1-\lambda}\right).
\end{equation}
By the mean-value theorem, and the fact that $\lambda\leq \lambda_0<1$, we have:
\begin{align}
\left|f_X\left(\frac{t-\lambda a}{1-\lambda}\right) - f_X(t)\right| & = \left|\frac{t-\lambda a}{1-\lambda}-t \right|\cdot |f_X'(\xi)|, \quad \quad \exists \xi \in I_{t,\lambda} \\
&\leq \frac{\lambda}{1-\lambda_0}|t-a|\sup_{\xi \in I_{t,\lambda}} |f_X'(\xi)|\label{eq:temp3}.
\end{align}
Now, we study the derivative $f_X'(t)$ of the log-normal density, which computes easily as:
$$
f_X'(t) = -\frac{\exp(-\frac18\log(t)^2) (4+\log t)}{8\sqrt{2\pi}t^2}.
$$
Note that, as $t\to 0$, $-(4+\log t) = \log(1/t)(1+o(1))$, and thus, as $t\to 0$,
$$
f_X'(t) = \frac{1+o(1)}{8\sqrt{2\pi}}\exp\left(-\frac18\log(1/t)^2 +\log(\log(1/t)) +2\log(1/t)\right)=o(1).
$$
A similar conclusion holds also as $t\to\infty$. Inspecting the graph of this function, we encounter the following features: 
\begin{itemize}
    \item $f_X'(t)\geq 0$ on $[0,e^{-4}]$, and $f_X(t)<0$ on $(e^{-4},\infty)$.  
    \item There exists a $T_1\in(0,e^{-4})$ , such that $f_X'(t)$ is increasing on $(0,e^{-4})$, and decreasing on $(T_1,e^{-4})$. 
    \item There exists a $T_2\in(e^{-4},\infty)$ such that, $f_X'(t)$ is decreasing on $(e^{-4},T_2)$, and is increasing on $(T_2,\infty)$. 
\end{itemize}
In particular, $\sup_{t\in \R}|f_X'(t)|\leq \max\{f_X(T_1),-f_X(T_2)\}\triangleq \mathcal{M}_2$ (an absolute constant), recalling that $f_X(T_2)<0$. Now, as long as $t\geq \max(a,T_2)$, and recalling Equation (\ref{eq:I-t-lambda}), since $|f_X'(t)|$ is decreasing on $I_{t,\lambda} = (t,\frac{t-\lambda}{1-\lambda})$ (since $f_X'(t)$ is increasing, and negative on this interval, we have the aforestated condition for $|f_X'(t)|$), we have that $\sup_{\xi \in I_{t,\lambda}}|f_X'(\xi)| = |f_X'(t)|$. We now upper bound the integral:
$$
\int_{\lambda a}^\infty \left|f_X\left(\frac{t-\lambda a}{1-\lambda}\right)-f_X(t)\right|\;dt, 
$$
by splitting into two pieces: $t\in[\lambda a,\max(a,T_2)]$, and $t\in(\max(a,T_2),\infty)$. Recalling Equation (\ref{eq:temp3}), we have:
$$
\int_{\lambda a}^\infty \left|f_X\left(\frac{t-\lambda a}{1-\lambda}\right)-f_X(t)\right|\;dt\leq \frac{\lambda}{1-\lambda_0}\int_{\lambda a}^{\infty}|t-a|\sup_{\xi \in I_{t,\lambda}} |f_X'(\xi)|\; dt.
    $$
Now, investigating right-hand-side, we have:
\begin{align*} 
&\frac{\lambda}{1-\lambda_0}\left(\int_{\lambda a}^{\max(a,T_2)}|t-a|\sup_{\xi \in I_{t,\lambda}}|f_X'(\xi)|\;dt  + \int_{\max(a,T_2)}^\infty |t-a|\sup_{\xi \in I_{t,\lambda}}|f_X'(\xi)|\; dt\right) \\
&\leq \frac{\lambda}{1-\lambda_0}\left(\int_{\lambda a}^{\max(a,T_2)} |t-a|\mathcal{M}_2\; dt + \int_{\max(a,T_2)}^\infty |t-a|\cdot |f_X'(t)|\; dt\right) \\
&\leq \frac{\lambda}{1-\lambda_0}\left(\mathcal{C}_1(a)+ \int_{\max(a,T_2)}^\infty |t-a|\cdot |f_X'(t)|\; dt\right),
\end{align*}
using the fact that, $\int_{\lambda a}^{\max(a,T_2)}|t-a|\mathcal{M}_2\; dt$ is upper bounded by some absolute constant $\mathcal{C}_1(a)$, depending only on $a$ (by simply considering integral from $0$ to avoid $\lambda$ dependency, and the fact that $\mathcal{M}_2$ is finite). For the second integral, observe that:
\begin{align*}
\int_{\max(a,T_2)}^\infty |t-a|\cdot |f_X'(t)|;dt &= \int_{\max(a,T_2)}^\infty (t-a)\cdot \frac{\exp(-\frac18\log(t)^2)(4+\log(t))}{8\sqrt{2\pi}t^2}\;dt  \\
&\leq \frac{1}{8\sqrt{2\pi}}\int_{\max(a,T_2)}^\infty \frac{\exp(-\frac18\log(t)^2)(4+\log t)}{t}\; dt = \mathcal{C}_2(a)<\infty,
\end{align*}
using the fact that  the integrand is equal to,
$$
\exp\left(-\frac18\log(t)^2 + \log(4+\log(t))-\log(t)\right),
$$
which is 
$$
\exp\left(-\frac18 \log(t)^2 +O(\log t)\right),
$$ 
as $t\to\infty$. Combining these lines, we therefore have,
$$
\int_{\lambda a}^{\infty}\left|f_X\left(\frac{t-\lambda a}{1-\lambda}-f_X(t)\right|\right) \; dt \leq \frac{\lambda}{1-\lambda_0}(\mathcal{C}_1(a)+\mathcal{C}_2(a)),
$$
where $\mathcal{C}_1(a)$ and $\mathcal{C}_2(a)$ are two finite constants, depending only on $a$. Finally, recalling Equation (\ref{eq:tv-X-lambda-X-zero}), we then have:
\begin{align*}
d_{TV}(X(\lambda),X(0)) & \leq \lambda\left(\frac{1}{2(1-\lambda_0)}+\frac12\mathcal{M}_1\right) + \frac{\lambda}{2(1-\lambda_0)^2}(\mathcal{C}_1(a)+\mathcal{C}_2(a)) \\
& = \lambda \left(\frac{1}{2(1-\lambda_0)}+\frac12\mathcal{M}_1 + \frac{1}{2(1-\lambda_0)^2}(\mathcal{C}_1(a)+\mathcal{C}_2(a))\right)\\
&\triangleq \lambda\mathcal{C}_{\lambda_0}
\end{align*}
for every $\lambda\in[0,\lambda_0]$, as claimed earlier. Finally, taking $\mathcal{C}_{ij}=\max\{\mathcal{C}_{\lambda_0},1/\lambda_0\}$, we have $d_{TV}(X(\lambda),X(0))\leq \mathcal{C}_{ij}\lambda$ for every $\lambda\in[0,1]$. 

\qedhere
\end{proof}
\subsection{Proof of Lemma \ref{lemma:tv-tensorize}}\label{pf:lemma-tv-tensorize}
\begin{proof}
Recall the following coupling interpretation of total variation distance:
$$
d_{TV}(P,Q) = \inf\{\mathbb{P}(X\neq Y):(X,Y)\text{ is such that } X\distr P, Y\distr Q\}.
$$
Now, let $P_1,\dots,P_\ell$ and $Q_1,\dots,Q_\ell$ be measures defined on a sample space $\Omega$. Suppose $X_1,\dots,X_\ell$ are independent random variables with $X_i\distr P_i$ for $1\leqslant i\leqslant \ell$; and $Y_1,\dots,Y_\ell$ are independent random variables with $Y_i\distr Q_i$ for $1\leqslant i\leqslant \ell$. Consider the vectors, $\mathbf{X}=(X_1,\dots,X_\ell)$ and $\mathbf{Y}=(Y_1,\dots,Y_\ell)$. Observe that, $\mathbf{X}\distr \otimes_{k=1}^\ell P_k$ and $\mathbf{Y}\distr \otimes_{k=1}^\ell Q_k$. Note that, 
$$
\{\mathbf{X}\neq \mathbf{Y}\} \subseteq \bigcup_{k=1}^\ell \{X_k\neq Y_k\}.
$$
Now, using union bound, we have:
$$
d_{TV}\left(\otimes_{k=1}^\ell P_k , \otimes_{k=1}^\ell Q_k\right)\leqslant \mathbb{P}(\mathbf{X}\neq \mathbf{Y}) \leq \sum_{k=1}^\ell \mathbb{P}(X_k\neq Y_k).
$$
Now, recalling
$$
d_{TV}(P_k,Q_k) = \inf_{(X_k,Y_k):X_k\distr P_k,Y_k\distr Q_k}  \mathbb{P}(X_k\neq Y_k),
$$
and taking infimums on the right hand side, we immediately obtain:
$$
d_{TV}\left(\otimes_{k=1}^\ell P_k , \otimes_{k=1}^\ell Q_k\right)\leqslant \sum_{k=1}^\ell d_{TV}(P_k,Q_k),
$$
as claimed.
\end{proof}
\subsection{Proof of Lemma \ref{lemma:markov}}\label{subsec:lemma-markov}
\begin{proof}
Letting $Y_k=-X_k$ with $\mathbb{E}[Y_k]\leqslant-q$, we have:
\begin{align*}
    \mathbb{P}\left(\frac{1}{\ell}\sum_{k=1}^\ell X_k>\epsilon\right)  = 1-\mathbb{P}\left(\frac{1}{\ell}\sum_{k=1}^\ell Y_k\geqslant -\epsilon\right)
    =1-\mathbb{P}\left(\frac{1}{\ell}\sum_{k=1}^\ell (1+Y_k)\geqslant 1-\epsilon\right)\geqslant 1-\frac{1-q}{1-\epsilon},
\end{align*}
since for $Y=\frac{1}{\ell} \sum_{k=1}^\ell (1+Y_k)\geqslant 0$, it holds that $\mathbb{E}[Y]\leqslant 1-q$, and therefore by Markov inequality, we have $\mathbb{P}(Y\geqslant 1-\epsilon)\leqslant \frac{\mathbb{E}[Y]}{1-\epsilon}\leqslant \frac{1-q}{1-\epsilon}$.
\end{proof}

\bibliographystyle{amsalpha}


\newcommand{\etalchar}[1]{$^{#1}$}
\providecommand{\bysame}{\leavevmode\hbox to3em{\hrulefill}\thinspace}
\providecommand{\MR}{\relax\ifhmode\unskip\space\fi MR }
\providecommand{\MRhref}[2]{%
  \href{http://www.ams.org/mathscinet-getitem?mr=#1}{#2}
}
\providecommand{\href}[2]{#2}

\end{document}